\documentclass[twocolumn,amsmath,amssymb,aps, physrev,onecolumn,preprint]{revtex4-1}

\usepackage{color}
\usepackage[framemethod=tikz]{mdframed}

\usepackage{graphicx}
\usepackage{amsmath}
\usepackage{amssymb}
\usepackage{algorithm}
\usepackage[update,prepend]{epstopdf}
\usepackage{subfigure}
\usepackage{amsthm}
\usepackage{enumerate}
\usepackage{color}
\usepackage{chemarr}
\usepackage{siunitx}

\usepackage{setspace}

\usepackage[version=3]{mhchem}

\def\fatx{\mathbf{x}}

\def\fatnu{\mathbf{\nu}}

\def\fatmu{\boldsymbol{\mu}}

\def\fatx{\mathbf{x}}

\def\faty{\mathbf{y}}
\def\fatY{\mathbf{Y}}

\def\fatnu{\boldsymbol{\nu}}

\def\voxsize{h}
\def\voxelsize{h}

\def\numvoxels{N}

\def\rradius{\sigma}
\def\diffconst{D}

\def\krmi{k_a}
\def\kdmi{k_d}
\def\krme{k_a^{\mathrm{meso}}}

\def\krmenl{k_a^{\mathrm{meso}}}

\def\kdme{k_d^{\mathrm{meso}}}
\def\kme{k_{\mathrm{CK}}}

\def\hstarnl{h^*_{\krmi,g}}
\def\hstarinfnl{h^*_{\infty,g}}
\def\hstarinf{h^*_{\infty}}

\def\mesobind{\tau_{\mathrm{meso}}}
\def\microbind{\tau_{\mathrm{micro}}}
\def\mesobindnl{\tau_{\mathrm{meso,r}}}
\def\mesobindloc{\tau_{\mathrm{meso}}}
\def\mesorebind{\tau_{\mathrm{meso}}^{\mathrm{rebind}}}
\def\microrebind{\tau_{\mathrm{micro}}^{\mathrm{rebind}}}
\def\mesorebindnl{\tau_{\mathrm{meso,r}}^{\mathrm{rebind}}}
\def\mesorebindloc{\tau_{\mathrm{meso}}^{\mathrm{rebind}}}
\def\mesorebindr{\tau_{\mathrm{meso,r}}^{\mathrm{rebind}}}
\def\mesobindr{\tau_{\mathrm{meso,r}}}
\newcommand{\mesobindrv}[1]{\tau_{\mathrm{meso,#1}}}
\newcommand{\mesorebindrv}[1]{\tau_{\mathrm{meso,#1}}^{\mathrm{rebind}}}

\def\Gdfull{G^{(d)}(\voxelsize,\rradius)}

\def\localrated{\rho^{(d)}}

\def\localratedfull{\localrated(\krmi,\voxsize)}

\def\mesobindlocal{\mesobind^{\rho}}
\def\mesorebindlocal{\tau_{\mathrm{meso}}^{\mathrm{rebind},\rho}}
\def\timejump{t_j}

\newcommand{\taunl}[1]{\tau_{#1}^g}
\newcommand{\tauloc}[1]{\tau_{#1}}

\newcommand{\numsteps}[2]{M_{#1}^{#2}}
\newcommand{\step}[1]{M_s^{#1}}
\newcommand{\revreaction}[5]{#1+#2\overset{#3}{\underset{#4}{\rightleftarrows}}#5}

\newtheorem{theorem}{Theorem}
\newtheorem{lemma}{Lemma}

\begin{document}

\begin{abstract}
It has been established that there is an inherent limit to the accuracy of the reaction-diffusion master equation.
Specifically, there exists a fundamental lower bound on the mesh size, below which the accuracy deteriorates as the mesh is refined further. In this paper we extend the standard reaction-diffusion master equation to allow molecules occupying neighboring voxels to react, in contrast to the traditional approach in which molecules react only when occupying the same voxel. We derive reaction rates, in two dimensions as well as three dimensions, to obtain an optimal match to the more fine-grained Smoluchowski model, and show in two numerical examples that the extended algorithm is accurate for a wide range of mesh sizes, allowing us to simulate systems intractable with the standard reaction-diffusion master equation. In addition, we show that for mesh sizes above the fundamental lower limit of the standard algorithm, the generalized algorithm reduces to the standard algorithm. We derive a lower limit for the generalized algorithm, which, in both two dimensions and three dimensions, is on the order of the reaction radius of a reacting pair of molecules.  

\end{abstract}

\title{Reaction rates for a generalized reaction-diffusion master equation}

\author{Stefan Hellander}
\affiliation{Department of Computer Science, University of California,Santa Barbara, CA 93106-5070 Santa Barbara, USA.}
\author{Linda Petzold}
\affiliation{Department of Computer Science, University of California,Santa Barbara, CA 93106-5070 Santa Barbara, USA.}
\maketitle

\maketitle


\section{Introduction}

Stochastic modeling has become a ubiquitous tool in the study of biochemical reaction networks \cite{Lawson:2013,Howard,FaEl,Sturrock1:2013,Sturrock2:2013}, as the traditional approach of deterministic modeling has been shown to be unsuitable for some systems where species are present in low copy numbers, or systems with spatial inhomogeneities \cite{TaTNWo10,FaEl}. Instead stochastic, spatially homogenous or inhomogeneous, models are employed.

Stochastic modeling can be carried out on multiple different scales. For processes occurring on the time scales typical of living cells we consider three different modeling scales: the spatially homogeneous well-mixed scale, the mesoscopic spatially heterogeneous scale, and the microscopic particle-tracking scale. In this paper the focus is on spatially heterogeneous modeling.

A prevalent model on the mesoscopic scale is the standard reaction-diffusion master equation (RDME), in which diffusion of individual molecules is modeled by discrete jumps between voxels, while reactions occur with a given intensity once molecules occupy the same voxel. The next subvolume method (NSM) \cite{ElEh04} is an efficient algorithm for generating single trajectories of the system. The NSM has been implemented in several software packages, including URDME \cite{URDME_BMC}, PyURDME (www.pyurdme.org), STEPS \cite{steps}, and MesoRD \cite{mesoRD}. It is also available as a part of larger simulation frameworks such as StochSS (www.stochss.org) and E-Cell \cite{Tomita99}.

On the microscopic scale we model the molecules as hard spheres moving by normal diffusion. We track the continuous position of individual molecules, and molecules react with a probability upon collision. This model is commonly referred to as the Smoluchowski model \cite{Smol}, with the addition of a Robin boundary condition at the reaction radius of a pair of molecules. Algorithms aimed at accurately and efficiently simulating the Smoluchowski model for general systems have been implemented in E-Cell \cite{Tomita99}, Smoldyn \cite{AnAdBrAr10}, and MCell \cite{MCell08}.

It has previously been shown that there is an inherent bound of several reaction radii on the spatial accuracy of the RDME compared to the Smoluchowski model \cite{HHP,HHP2}. It was shown in \cite{HHP2} that by choosing correct mesoscopic reaction rates, the RDME could be made accurate all the way down to this lower bound. However, for mesh resolutions below this lower bound, the accuracy deteriorates. 

In this paper we generalize the standard RDME by letting molecules occupying neighboring voxels react. Henceforth we refer to this generalization as the \emph{generalized RDME}. Similar generalizations have been considered previously in \cite{IsaacsonCRDME,FangeSRDME}. In \cite{IsaacsonCRDME}, Isaacson discretizes the Doi model \cite{Doi1} to obtain a convergent RDME. In \cite{FangeSRDME}, reaction rates are derived for a spherical model and applied to the RDME on a Cartesian mesh. In this paper we take a fundamentally different approach. By deriving reaction rates to match certain statistics of the Smoluchowski model, we arrive at analytical expressions for the reaction rates, and show that this approach yields accurate results down to a fundamental lower limit on the mesh size. This mesh size will be on the order of the reaction radius of two molecules.

Importantly, we derive reaction rates under specific assumptions on the dynamics of dissociating molecules, and we show with a simple example that not doing so may lead to reaction rates that are inaccurate for certain systems. We thus argue that it is crucial to take dissociations into account in the derivation of reaction rates for the generalized RDME.

The outline of the paper is as follows. In Section \ref{sec:background} we review the Smoluchowski model and the RDME, and how they are connected through the mesoscopic reaction rates. In Section \ref{sec:theory} we describe the generalized algorithm, and derive accurate mesoscopic reaction rates as well as the lower limit on the mesh size. Finally, in Section \ref{sec:results}, we study two numerical examples, demonstrating the accuracy of the generalized RDME, and how it can be used to simulate systems that are intractable with the standard RDME.

\section{Background}
\label{sec:background}

\subsection{Microscopic level}

At this level of modeling we track the continuous position of individual molecules moving by normal diffusion. Each species $S_i$ has a diffusion constant $D_i$ and a reaction radius $\sigma_i$. Consider two molecules, one of species $S_1$ and one of species $S_2$, with positions $\fatx_{1n}$ and $\fatx_{2n}$ at time $t_n$. The molecules can react according to $S_1+S_2\overset{\krmi}{\underset{\kdmi}{\rightleftarrows}} S_3$, where $S_3$ is some product. The probability distribution function (PDF) $p(\fatx_1,\fatx_2,t|\fatx_{1n},\fatx_{2n},t_n)$ represents the probability that the positions of the molecules are given by $\fatx_1$ and $\fatx_2$ at time $t$; $p$ then solves the Smoluchowski equation
\begin{align}
\partial_t p = D_1\Delta_{\fatx_1}p+D_2\Delta_{\fatx_2}p.
\end{align}
It can be shown that with the change of variables
\begin{align}
&\fatY = \sqrt{D_2}{D_1}\fatx_1+\sqrt{D_1}{D_2}\fatx_2\\
&\faty = \fatx_2-\fatx_1,
\end{align}
we obtain two independent equations, where the equation for $\fatY$ describes free diffusion, while the equation for $\faty$ becomes
\begin{align}
\label{eq:smolu1}
\partial_t p(\faty,t) = D\Delta_{\faty}p,
\end{align}
where $D=D_1+D_2$. Let $\rradius = \sigma_1+\sigma_2$ be the sum of the reaction radii. We introduce a reactive Robin boundary condition at $\rradius$
\begin{align}
\label{eq:smolu2}
K\frac{\partial p_{\faty}}{\partial n}\bigg|_{\| \faty \|=\rradius} = \krmi p_{\faty}(\| \faty\| = \rradius,t),
\end{align}
where
\begin{align}
K = \begin{cases}
2\pi\rradius D,\quad (2D)\\
4\pi\rradius^2D,\quad (3D),
\end{cases}
\end{align}
and where $\krmi$ is the microscopic reaction rate. The initial condition is given by $p_{\faty}(\faty,t_n) = \delta(\faty-\faty_n)$, and, since we assume that there is no outer boundary, we enforce $p_{\faty}(\| \faty \| \to\infty,t) = 0$. This equation can be solved analytically \cite{CarJae}, but the solution is difficult and expensive to evaluate numerically. Applying an operator split method to \eqref{eq:smolu1}-\eqref{eq:smolu2} can significantly simplify the process of sampling new positions from the PDF \cite{SHeLo11}.

An $S_3$ molecule is assumed to dissociate according to an exponential distribution with mean $\kdmi$. Following a dissociation, the two products $S_1$ and $S_2$ are placed in contact a distance of $\sigma$ apart.

A system of more than two molecules is not amenable to the direct approach of solving for the full PDF, due to the high dimensionality of the problem. A common approach is instead to approximate the full problem as a set of one- and two-body problems, by dividing the system into subsets of single and pairs of molecules according to the distances between them. We can obtain a good approximation of the full problem by updating each subset independently during short time steps $\Delta t$. This algorithm is commonly referred to as Green's function reaction dynamics (GFRD) \cite{ZoWo5a,ZoWo5b}. All microscale computations in this paper are carried out using a variant of the GFRD algorithm, described in \cite{SHeLo11}.

\subsection{Reaction-diffusion master equation}

At the mesoscopic scale the simulation domain is discretized by $N$ non-overlapping voxels, and diffusion is modeled as discrete jumps between the nodes of the voxels. The mesh may be either a Cartesian mesh, or an unstructured, tetrahedral (3D) or triangular (2D), mesh. A Cartesian mesh is suitable if the domain is simple, for instance a square or a cube, while an unstructured mesh has advantages for complicated domains. The jump coefficients between voxels are given by $h^2/(2dD)$ in the case of a Cartesian mesh, where $h$ is the width of a voxel, $d$ the dimension, and $D$ the diffusion rate of the molecule. For an unstructured mesh, the jump coefficients can be obtained from a finite element discretization of the diffusion equation \cite{uRDME}. Reactions occur with some intensity when molecules occupy the same voxel. 

Let $p(\fatx,t|\fatx_n,t_n)$ be the probability that the system is found in state $\fatx$ at time $t$, given that it was in state $\fatx_n$ at time $t_n$. For brevity of notation, let $p(\fatx,t) = p(\fatx,t|\fatx_n,t_n)$. Let $\fatx_{i\cdot}$ and $\fatx_{\cdot j}$ denote the $i$-th row and the $j$-th column of the $K\times S$ state matrix $\fatx$, respectively, where $S$ is the number of species of the system. 
The RDME is given by
\begin{align}
\label{eq:rdme}
\frac{\mathrm{d}}{\mathrm{dt}}p(\fatx, t) = 
&\sum_{i=1}^{N}
\sum_{r = 1}^{M}
a_{ir}(\fatx_{i \cdot}-\fatmu_{ir})p(\fatx_{1 \cdot},\ldots,\fatx_{i \cdot}-\fatmu_{ir},
\ldots,\fatx_{N \cdot}, t) \nonumber 
-\sum_{i=1}^{N}
\sum_{r = 1}^{M}
a_{ir}(\fatx_{i \cdot})p(\fatx, t)\\
&+\sum_{j=1}^{S} \sum_{i = 1}^{N} \sum_{k=1}^N d_{jik}(\fatx_{\cdot j}-\fatnu_{ijk})
p(\fatx_{\cdot 1},\ldots,\fatx_{\cdot j}-\fatnu_{ijk},
\ldots,\fatx_{\cdot S}, t) \nonumber\\
&-\sum_{j=1}^S\sum_{i=1}^{N}
\sum_{k = 1}^{N} d_{ijk}(\fatx_{\cdot j})p(\fatx, t),\nonumber\\
\end{align}
where the propensity functions of the $M$ chemical reactions are denoted by $a_{ir}(\fatx_i)$, $\fatmu_{ir}$ are the stoichiometry vectors associated with the reactions, $d_{ijk}$ are the jump coefficients, and $\fatnu_{ijk}$ are stoichiometry vectors for diffusion events.

The RDME is in general too high-dimensional to be solved by direct approaches. An alternative approach is to generate individual trajectories of the system with stochastic simulations. The NSM \cite{ElEh04} is an efficient algorithm frequently used for that purpose.

\subsection{Reaction rates for the standard reaction-diffusion master equation}

Consider a system of two molecules, one of species $A$ and one of species $B$, that react according to $A+B\overset{\krmi}{\underset{\kdmi}{\rightleftarrows}} C$, where $\krmi$ and $\kdmi$ are the microscopic reaction rates. Assume that the molecules diffuse in a square (2D) or cube (3D) with periodic boundary conditions. Without loss of generality, assume that the $B$ molecule is fixed at the origin, and that the $A$ molecule diffuses freely with a diffusion rate $D$. The $A$ molecule is initialized according to a uniform distribution.

Let $\mesobind(\krme,h)$ be the mean association time of the two molecules on the mesoscopic scale, and let $\microbind(\krmi)$ be the mean association time on the microscopic scale. Under the assumption that $\mesobind(\krme,h)=\microbind(\krmi)$ holds, it is shown in \cite{HHP,HHP2} that the mesoscopic association rate is given by
\begin{align}
\label{eq:local-meso-rate}
\krme = \localratedfull = \frac{\krmi}{h^d}\left[ 1+\frac{\krmi}{D}\Gdfull\right]^{-1},
\end{align}
where $d$ is the dimension,
\begin{align}
\Gdfull = \begin{cases}
\frac{1}{2\pi}\log\left(\pi^{-\frac{1}{2}}\frac{h}{\rradius}\right)-\frac{1}{4}\left(\frac{3}{2\pi}+C_2\right)\quad \mathrm{(2D)}\\
\frac{1}{4\pi\sigma}-\frac{C_3}{6\voxsize}\quad \mathrm{(3D)},
\end{cases}
\end{align}
and
\begin{align}
\label{eq:C_consts}
C_d\approx \begin{cases}
0.1951,\quad d=2\\
1.5164,\quad d=3.
\end{cases}
\end{align}
The microscopic parameters are $\rradius$, the sum of the reaction radii of the molecules, $\diffconst$, the sum of the diffusion constants, and $\krmi$, the microscopic reaction rate. To simplify the notation somewhat, we let $\mesobindlocal(\krmi,\voxsize):=\mesobind(\localrated(\krmi,\voxsize),\voxsize)$. For a reversible reaction we match the mean binding time for $h>\hstarinf$, where
\begin{align}
\label{eq:local-h-star-inf}
\hstarinf \approx \begin{cases}
\sqrt{\pi}\exp\left(\frac{3+2\pi C_2}{4}\right)\sigma\approx 5.1\sigma,\quad (2D)\\
\frac{2}{3}\pi C_3\sigma\approx 3.2\sigma,\quad (3D).
\end{cases}
\end{align}

Let $\mesorebind(\krme,\voxsize)$ and $\microrebind(\krmi)$ denote the average rebinding times---that is, the average time until two molecules react, given that they have just dissociated---on the mesoscopic and microscopic scale, respectively. Again, to simplify notation, we let $\mesorebindlocal(\krmi,\voxsize):=\mesorebind(\localrated(\krmi,\voxsize),\voxsize)$. The rebinding times can be written in terms of the average binding times
\begin{align}
&\mesorebindlocal(\krmi,\voxsize) = \mesobindlocal(\krmi,\voxsize)-\mesobindlocal(\infty,\voxsize)\label{eq:rebind1-meso}\\
&\microrebind(\krmi) = \microbind(\krmi)-\microbind(\infty),\label{eq:rebind1-micro}
\end{align}
where, for simplicty of notation, $\mesobindlocal(\krmi\to\infty,\voxsize)$ and $\microbind(\krmi\to\infty)$ are denoted by $\mesobindlocal(\infty,h)$ and $\microbind(\infty)$, respectively. That \eqref{eq:rebind1-meso} and \eqref{eq:rebind1-micro} hold can be realized by considering the following argument. Given a uniform initial distribution, $\mesobind(\infty)$ is the time until the molecules are in the same voxel for the first time. By subtracting that time from the total binding time, we obtain the rebinding time. A similar argument holds for the microscopic case. We immediately see that because $\mesobindlocal(\krmi,\voxsize)=\microbind(\krmi)$ holds, the rebinding times will match if and only if $\mesobindlocal(\infty,\voxsize)=\microbind(\infty)$. This holds for $h=\hstarinf$, and consequently
\begin{align}
&\mesorebindlocal(\krmi,h)>\microrebind(\krmi) \text{ for }\voxsize>\hstarinf\\
&\mesorebindlocal(\krmi,h)=\microrebind(\krmi) \text{ for }\voxsize=\hstarinf\\
&\mesorebindlocal(\krmi,h)<\microrebind(\krmi) \text{ for }\voxsize<\hstarinf.
\end{align}
As a mesoscopic dissociation event is a combination of microscopic dissociation and the diffusion required to get well mixed in a voxel, we require that $\mesorebind\geq \microrebind$ should hold. For $h<\hstarinf$ we cannot match the mean binding time while satisfying $\mesorebind\geq \microrebind$, and the accuracy of the RDME consequently deteriorates with decreasing $\voxsize$. Thus, $\hstarinf$ is the finest spatial resolution attainable with the standard RDME.

For a given $h>\hstarinf$, we can compute the error in rebinding time as
\begin{align}
\label{eq:rebind-loc-error}
\left| \mesorebindlocal(\krmi,h)-\microrebind(\krmi) \right| = \left| \mesobind(\infty,\voxsize)-\microbind(\infty) \right|,
\end{align}
where the right-hand side thus is a measure of how well-resolved a system is. Details of the above theory can be found in \cite{HHP2}.

\section{The generalized reaction-diffusion master equation}
\label{sec:theory}

In the standard RDME, molecules react only when they occupy the same voxel. In this section we extend this approach by allowing molecules occupying neighboring voxels to react. To connect the standard RDME to the microscopic Smoluchowski model, we determined the rate with which molecules react when occupying the same voxel. For the generalized RDME we need to obtain the rates for molecules occupying the same voxel, but also the rates for molecules occupying neighboring voxels. In \cite{HHP,HHP2} we derive rates for the standard RDME by matching the mean association times on the two scales. To uniquely determine both of the rates for the generalized RDME, we need an additional constraint.

In Sec.~\ref{sec:algorithm} we outline the algorithm, and in Sec.~\ref{sec:reaction-rates} we derive mesoscopic parameters by trying to match certain statistics of the microscopic model to the corresponding statistics on the mesoscopic scale. In Sec.~\ref{sec:dissociation} we determine the dissociation rate of a reversibly reacting pair of molecules, and in Sec.~\ref{sec:algsummary} we collect the results and summarize the algorithm.

\subsection{Generalized reactions}
\label{sec:algorithm}

Consider a domain $\Omega$ discretized by a Cartesian mesh, and a single reversible reaction $\revreaction{A}{B}{\krmi}{\kdmi}{C}$. We extend the generalized RDME to allow reactions between molecules occupying neighboring voxels. Thus, if a molecule of species $A$ occupies the same voxel as a molecule of species $B$, they react with an intensity given by $k_0$. If the molecules instead occupy neighboring voxels they react with an intensity of $k_1$, where two voxels are neighbors if they share one side.

We can choose $k_0$ and $k_1$ freely, with the restriction that the total intensity should be constant. Call the total intensity $\krme$. Let $d$ be the dimension. Then, since each voxel has $2d$ neighbors, $k_0$ and $k_1$ must satisfy
\begin{align}
\label{eq:total_intensity_constant}
k_0+2dk_1 = \krme.
\end{align}
Thus we can write 
\begin{align}
\label{eq:r_and_krmeso}
&k_0 = (1-2dr)\krme \\
&k_1 = r\krme,
\end{align}
where $0\leq r \leq 1/(2d)$.

Now assume that a molecule of species $C$ dissociates. We must determine where to place the two products $A$ and $B$. It may seem natural to place them in the same voxel with probability $1-2dr$, and in neighboring voxels with probability $2dr$. While this arguably would yield the most accurate results compared to microscopic simulations for a single reversible reaction, we show below that this approach is unsuitable in general.

First consider the single irreversible dissociation given by
\begin{align}
\label{eq:irr_dissociation}
P\xrightarrow{k_{deg}}S_1+S_2.
\end{align}
In this case the microscopic and mesoscopic rates will be the same; thus $\kdme=k_{deg}$, and the products are placed in the same voxel. Now consider that in addition to \eqref{eq:irr_dissociation} we have the following reactions:
\begin{align}
S_1{\overset{k^*}{\rightarrow}}S_1^*\label{eq:add_reactions1}\\
S_1^*+S_2\overset{\krmi}{\underset{\kdmi}{\rightleftarrows}}S_3\label{eq:add_reactions2}.
\end{align}
Again, \eqref{eq:add_reactions1} is an irreversible unimolecular reaction and thus the mesoscopic and microscopic rates are the same. Now, if $k^*$ is large, the system \eqref{eq:irr_dissociation}-\eqref{eq:add_reactions2} will be well approximated by
\begin{align}
\label{eq:reduced_system}
P\xrightarrow{k_{deg}}S_1^*+S_2\overset{\krmi}{\underset{\kdmi}{\rightleftarrows}}S_3.
\end{align}
Had we derived rates for reaction \eqref{eq:add_reactions2} assuming that dissociating molecules are placed in neighboring voxels with some probability, we can see that the sequence \eqref{eq:reduced_system} will be incorrectly simulated, as $S_1$ and $S_2$ are placed in the same voxel with probability 1 when $P$ dissociates. Specifically, the rebinding dynamics of $S_1^*$ and $S_2$ will be incorrect, as the rebinding time will depend on whether they were produced from a dissociating $S_3$ or a dissociating $P$.

To summarize:
\begin{itemize}
\item Reactive molecules occupying the same voxel react with intensity $(1-2dr)\krme$.
\item Reactive molecules in neighboring voxels react with intensity $r\krme$.
\item When a molecule dissociates, the products are placed in the same voxel with probability 1.
\end{itemize}

The parameters $r$ and $\krme$ now have to be determined from the microscopic parameters $\krmi$, $\sigma$ and $D$.

\subsection{Reaction rates}
\label{sec:reaction-rates}

Consider the reversible reaction $A+B\overset{\krmi}{\underset{\kdmi}{\rightleftarrows}}C$. Assume that the initial state of the system is given by one molecule of species $A$ and one molecule of species $B$ in a square (2D) or a cubic (3D) domain $\Omega$ of width $L$ with periodic boundary conditions. For simplicity, and without loss of generality, assume that the $A$ molecule is fixed at the origin while the $B$ molecule has a uniform initial distribution and a diffusion rate $D=D_A+D_B$. On the microscopic scale the $B$ molecule moves by continuous Brownian motion. On the mesoscopic scale, $\Omega$ is subdivided into non-overlapping squares or cubes of width $h$. The $B$ molecule thus jumps between voxels with an intensity of $k_j = 2dD/h^2$, with $d$ the dimension. 
Let $\mesobindr(\krme,h)$ denote the average time until the molecules react on the mesoscopic scale with the generalized RDME.

For the local RDME, it was shown in \cite{HHP,HHP2} that by enforcing the constraint $\mesobind=\microbind$ we obtain mesoscopic reaction rates as given by \eqref{eq:local-meso-rate}. In addition, it was shown that $\mesorebind$ approaches $\microrebind$ from above as $\voxsize\to\hstarinf$. Therefore it seems reasonable to require that, with the generalized RDME, we obtain an approximation of $\microrebind$ that is equal to or better than the approximation we obtain with the standard RDME. The first constraint is therefore that the mean binding time agrees between the mesoscopic and the microscopic scales
\begin{align}
\mesobindr(\krme,h) = \microbind(\krmi), \label{eq:nonlocal-constraint1}
\end{align}
and the second constraint will be that, given \eqref{eq:nonlocal-constraint1}, $\krme$ and $r$ minimizes the difference between the rebinding times at the mesoscopic and microscopic scales; that is, we want to minimize
\begin{align}
\left| \mesorebindr(\krme,h)-\microrebind(\krmi) \right|,\label{eq:nonlocal-constraint2}
\end{align}
under the assumption that \eqref{eq:nonlocal-constraint1} holds, where $\mesobindr$ is the average binding time (dependent on $r$ and $\krme$), and where $\mesorebindr$ is the average rebinding time, in the generalized RDME. Note that with \eqref{eq:nonlocal-constraint1} satisfied we have $\mesobindrv{0}(\krme,h) = \mesobind(\krme,h)$ and $\mesorebindrv{0}(\krme,h) = \mesorebind(\krme,h)$.

\subsubsection{Mean mesoscopic binding time}
Again, assume that we have species $A$ and $B$, with one molecule of each, and that the $A$ molecule is fixed. The $B$ molecule is initialized according to a uniform distribution, and diffuses with diffusion rate $D$.

We start by deriving the mesoscopic mean binding time. To this end, let $\step{i}$ denote the average number of diffusive jumps required for the $B$ molecule to reach a voxel at distance $i$ from the $A$ molecule, where the distance between two voxels is defined to be the smallest number of discrete jumps required to move from one voxel to the other. Let the set of all voxels at a distance $i$ from the $A$ molecule be denoted by $d_i$, let $\timejump $ denote the average time for a diffusive jump, and let $\tau_i$ denote the average time for the $B$ and the $A$ molecule to react, given that the $B$ molecule is occupying a voxel at distance $i$ from the $A$ molecule. Thus, $\timejump  = h^2/(2dD)$, and
\begin{align}
\label{eq:tau_meso_simple}
\mesobindr(\krme,\voxsize) = \step{1}\timejump +\tau_1.
\end{align}
The first term, $\step{1}\timejump $, represents the average time required for the $B$ molecule to reach $d_1$. The second term, $\tau_1$, represents the remaining time until the molecules react, given that the $B$ molecule occupies a voxel in $d_1$. To obtain $\mesobindr$ we now derive analytical expressions for $\step{1} $ and $\tau_1$.

\begin{lemma}
Let $N$ be the total number of voxels in the mesh. Then
\begin{align}
\label{eq:ns1}
\step{1}  = \begin{cases}
\pi^{-1}N\log{N}+(C_2-1)N+O(1)\quad \mathrm{(2D)}\\
(C_3-1)N+O(\sqrt{N})\quad \mathrm{(3D)},
\end{cases}
\end{align}
where $C_2$ and $C_3$ are defined in \eqref{eq:C_consts}.
Let $\numsteps{i}{j} $ denote the average number of steps required to diffuse from $d_i$ to $d_j$. We have
\begin{align}
\label{eq:n21}
\numsteps{2}{1}  = \frac{N-2}{2d-1}.
\end{align}
\end{lemma}
\begin{proof}
First note that 
\begin{align}
\label{eq:ns1_2}
\step{1}  = \step{0} -\numsteps{1}{0} . 
\end{align}
In \cite{Montroll68} it is shown that
\begin{align}
\label{eq:ns0}
\step{0}  = \begin{cases}
\pi^{-1}N\log{N}+C_2N+O(1)\quad\mathrm{(2D)}\\
C_3N+O(\sqrt{N})\quad\mathrm{(3D)}.
\end{cases}
\end{align}
Let $\numsteps{0}{0} $ be the average number of steps required to return to $d_0$, given that we start in $d_0$. The first jump of a molecule starting in $d_0$ always transfers the molecule to $d_1$, so we find that 
\begin{align}
\label{eq:n10}
\numsteps{1}{0}  = \numsteps{0}{0} -1.
\end{align} 
We know that $\numsteps{0}{0}  = N$, shown in \cite{MoWe65}. By combining \eqref{eq:ns1_2}, \eqref{eq:ns0}, and \eqref{eq:n10}, we obtain \eqref{eq:ns1}.

To obtain \eqref{eq:n21} we note that we can write $\numsteps{0}{0} $ as
\begin{align}
\label{eq:n00expanded}
\numsteps{0}{0}  = 1+\frac{1}{2d}+\frac{2d-1}{2d}\left( \numsteps{2}{1} +\numsteps{1}{0}  \right).
\end{align}
To see that the above equality holds, start by considering a molecule in $d_0$. The first jump transfers the molecule to $d_1$; the second jump transfers it back to the origin with a probability of $1/(2d)$, or to $d_2$ with a probability of $(2d-1)/(2d)$. The average number of steps required to reach $d_0$ from $d_2$, is given by the average number of steps to reach $d_1$ plus the average number of steps to reach $d_0$, given that the molecule starts in $d_1$. Now, solving \eqref{eq:n00expanded} for $\numsteps{2}{1} $ yields \eqref{eq:n21}.
\end{proof}

To obtain $\mesobindr$ for $\krmi<\infty$, it remains to determine $\tau_1$. To that end, assume that the $B$ molecule occupies a voxel in $d_1$, and that the intensity with which the molecules react in $d_1$ is given by $1/(r\krme)$. Then, to maintain a total intensity of $1/\krme$, the molecules must react with an intensity of $1/\left[(1-2dr)\krme\right]$ in $d_0$. We require that $r\geq 0$, and that $0 \leq 1-2dr \leq 1$. To simplify the notation we let $1/(r\krme)$ be denoted by $p_1$, and $1/\left[(1-2dr)\krme\right]$ by $p_0$.
\begin{lemma}
\label{lemma-tau1}
Let $\tau_1$ be the average time until the molecules react, given that the $B$ molecule occupies a voxel in $d_1$. Then
\begin{align}
\tau_1 &= \frac{(N+2d-1)p_0+(N+2d-2)\timejump }{2d(p_0+\timejump )+p_1}p_1\label{eq:tau1-th1}\\
&\approx\frac{p_0+\timejump }{p_0+2dr\timejump }\frac{N}{\krme}.\label{eq:tau1-th2}
\end{align}
\end{lemma}
\begin{proof}
Let $t_e^0$ and $t_e^1$ denote the average time until the next event fires, given that the $B$ molecule occupies a voxel in $d_0$ or $d_1$, respectively. Then $t_e^0 = 1/(p_0^{-1}+\timejump ^{-1})$ and $t_e^1 = 1/(p_1^{-1}+\timejump ^{-1})$. 

By assumption, the $B$ molecule initially occupies a voxel in $d_1$. The next event can either be: (1) a diffusive jump, with probability $p_1/(p_1+\timejump )$, or (2) a reaction event with probability $\timejump /(p_1+\timejump )$.

Now assume that the next event is a diffusion event. Then: (1.1) the molecule jumps to $d_2$ with probability $(2d-1)/2d$, or (1.2) the molecule jumps to $d_0$ with probability $1/(2d)$. Assume that the molecule jumps to $d_0$. Then the next event is: (1.2.1) a reaction with probability $\timejump /(p_0+\timejump )$, or (1.2.2) diffusion to $d_1$ with probability $p_0/(p_0+\timejump )$. Thus, if the molecule is in the state (1.2), the time until the molecules react is given by
\begin{align}
\label{eq:tau1p_tau0}
\tau_0 = \frac{\timejump }{p_0+\timejump }t_e^0+\frac{p_0}{p_0+\timejump }(t_e^0+\tau_1) = t_e^0+\frac{p_0}{p_0+\timejump }\tau_1.
\end{align}
Now instead assume that the molecule is in the state (1.1). The molecules cannot react until the $B$ molecule reaches $d_1$, and thus the average time until the molecules react is given by
\begin{align}
\label{eq:tau1p_tau2}
\tau_2 = \numsteps{2}{1} \timejump +\tau_1 = \frac{N-2}{2d-1}\timejump +\tau_1,
\end{align}
where $\numsteps{2}{1}$ is given by \eqref{eq:n21}. To summarize:
\begin{itemize}
\item[] The $B$ molecule initially occupies a voxel in $d_1$, and the $A$ molecule is fixed in $d_0$.
\item[(1)] The $B$ molecule diffuses with probability $p_1/(p_1+\timejump )$.
\begin{itemize}
\item[(1.1)] The $B$ molecule jumps to $d_2$ with probability $(2d-1)/(2d)$. The average remaining time until the $A$ and $B$ molecules react is given by $\tau_2$.
\item[(1.2)] The $B$ molecule jumps to $d_0$ with probability $1/(2d)$.
\begin{itemize}
\item[(1.2.1)] The $A$ and $B$ molecules react with probability $\timejump /(p_0+\timejump )$.
\item[(1.2.2)] The $B$ molecule diffuses to $d_1$ with probability $p_0/(p_0+\timejump )$. The average remaining time until the $A$ and $B$ molecules react is given by $\tau_1$.
\end{itemize}
\end{itemize}
\item[(2)] The $A$ and $B$ molecules react with probability $\timejump /(p_1+\timejump )$.
\end{itemize}
Putting it all together, we obtain
\begin{align}
\label{eq:tau1eq_small}
\tau_1 = \frac{\timejump }{p_1+\timejump }t_e^1+\frac{p_1}{p_1+\timejump }\left( t_e^1+\frac{1}{2d}\tau_0+\frac{2d-1}{2d}\tau_2 \right).
\end{align}
By inserting \eqref{eq:tau1p_tau0} and \eqref{eq:tau1p_tau2} into \eqref{eq:tau1eq_small} and solving for $\tau_1$ we obtain \eqref{eq:tau1-th1}, after some cumbersome but straightforward algebra. By assuming $\numvoxels\gg 1$, \eqref{eq:tau1-th2} follows.
\end{proof}
\begin{theorem}
Let $\mesobindr$ be the average time until the molecules react, given that the $B$ molecules have a uniform initial distribution. Then
\begin{align}
\label{eq:tau_meso_theorem}
\mesobindr(\krme,h) \approx \begin{cases}
\left[\pi^{-1}N\log{N}+(C_2-1)N\right]\timejump +\frac{p_0+t_j}{p_0+4rt_j}\frac{N}{\krme}   \quad \mathrm{(2D)}\\
(C_3-1)N\timejump +\frac{p_0+t_j}{p_0+6rt_j}\frac{N}{\krme}\quad \mathrm{(3D)}.
\end{cases}
\end{align}
\end{theorem}

\begin{proof}
This follows immediately from \eqref{eq:tau_meso_simple}, \eqref{eq:ns1} and \eqref{eq:tau1-th2}.
\end{proof}

It is of interest to know the smallest voxel size $\voxsize$ for which we can match the mesoscopic mean binding time, $\mesobindr$, with the microscopic mean binding time, $\microbind$. In \cite{HHP,HHP2} this problem was solved in the case of the standard RDME for a general reversible reaction $A+B \xrightleftharpoons[\kdmi]{\krmi} C$. Similar results in the case of the generalized RDME can be obtained for the case of an irreversible reaction with $\krmi\to\infty$. As we will find in Theorem \ref{fund_bound_th}, the lower bound for $\krmi\to\infty$ is in fact a fundamental lower bound for the generalized RDME.
\begin{theorem}
Let $\hstarnl$ be the smallest voxel size for which we can choose reaction rates such that $\mesobindr(\krme,h)=\microbind(\krmi)$. Then
\begin{align}
\label{eq:hstarinfnl}
\hstarinfnl=\begin{cases}
\sqrt{\pi}\exp\left( \frac{3+2\pi(C_2-1)}{4} \right)\rradius\approx 1.0599\rradius \quad \mathrm{(2D)} \\
\frac{2}{3}(C_3-1)\pi\rradius\approx 1.0815\rradius \quad \mathrm{(3D)}.
\end{cases}
\end{align}
\end{theorem}
\begin{proof}
We do not have analytical results for $\microbind$ on a square or a cube, but given that $L\gg\rradius$ is satisfied, an excellent approximation is provided by
\begin{align}
\microbind(\krmi) = \begin{cases}
\frac{1+\alpha F(\lambda)}{\krmi}L^2\quad\mathrm{(2D)}\\
\frac{L^3}{\kme}\quad\mathrm{(3D)},
\end{cases}
\end{align}
where
\begin{align}
\begin{split}
\lambda &= \pi^{\frac{1}{2}}\frac{\rradius}{L}\\
\alpha &= \frac{\krmi}{2\pi\diffconst} \\
F(\lambda) &= \frac{\log(1/\lambda)}{(1-\lambda^2)^2}-\frac{3-\lambda^2}{4(1-\lambda^2)},
\end{split}
\end{align}
and where $\kme=4\pi\rradius\diffconst\krmi/(4\pi\rradius\diffconst+\krmi)$ is the classical mesoscopic reaction rate, valid for large volumes, derived by Collins and Kimball in \cite{CollinsKimball}. The expression in 2D was derived in \cite{FBSE10}, following the approach devised in \cite{AgmonSzabo}.

Since molecules are allowed to react with molecules occupying neighboring voxels, we obtain
\begin{align}
\tau_1\to 0 \quad\mathrm{for}\quad \krmi\to\infty,\, r>0
\end{align}
and thus
\begin{align}
\mesobindr\to \step{1} \timejump \quad\mathrm{for}\quad\krmi\to\infty,\, r>0.
\end{align}
We know $\step{1} $ from \eqref{eq:ns1}, and we have $\timejump =h^2/(2d\diffconst)$ by definition. We now obtain \eqref{eq:hstarinfnl} by solving
\begin{align}
\label{eq:hstareq}
\step{1} \timejump  = \microbind(\infty)
\end{align}
for $\voxsize$.

In 3D, \eqref{eq:hstareq} becomes
\begin{align}
\label{eq:hstareq3D}
\frac{(C_3-1)L^3}{6Dh} = \frac{L^3}{4\pi\rradius\diffconst},
\end{align}
since $\step{1}  \sim (C_3-1)N$ for $N\gg 1$, and $\kme\to 4\pi\rradius\diffconst$ as $\krmi\to\infty$. Solving \eqref{eq:hstareq3D} for $\voxsize$ yields
\begin{align}
h=\frac{2}{3}(C_3-1)\pi\rradius\approx 1.0815\rradius.
\end{align}

In 2D, \eqref{eq:hstareq} becomes
\begin{align}
\label{eq:hstareq2D}
\frac{h^2}{4\diffconst}\left[ \pi^{-1}\frac{L^2}{h^2}\log\left( \frac{L^2}{h^2} \right) + (C_2-1)\frac{L^2}{h^2} \right] = \frac{L^2}{\krmi}+\frac{F(\lambda)}{2\pi\diffconst}L^2,
\end{align}
and for $\krmi\to\infty$, we have
\begin{align}
\label{eq:taumicro2Dlimit}
\frac{L^2}{\krmi}+\frac{F(\lambda)}{2\pi\diffconst}L^2\to\frac{\log\left( \pi^{-\frac{1}{2}}\frac{L}{\sigma} \right)-\frac{3}{4}}{2\pi\diffconst}L^2.
\end{align}
In \eqref{eq:taumicro2Dlimit} we used that $\lambda\approx 0$ for $L\gg\rradius$. Now \eqref{eq:hstareq2D} reduces to
\begin{align}
\pi^{-1}\log\left(\frac{L}{\voxsize}\right)+\frac{C_2-1}{2} = \pi^{-1}\log\left( \pi^{-\frac{1}{2}}\frac{L}{\sigma} \right) - \frac{3}{4\pi}.
\end{align}
We can rewrite the equation above to get
\begin{align}
\pi^{-1}\log\left( \pi^{\frac{1}{2}} \frac{\rradius}{\voxsize} \right) = \frac{1-C_2}{2}-\frac{3}{4\pi}.
\end{align}
Solving for $\voxsize$ yields
\begin{align}
h = \sqrt{\pi}\exp\left(\frac{3+2\pi(C_2-1)}{4}\right)\sigma\approx 1.0599\rradius
\end{align}
\end{proof}

\subsubsection{Mean mesoscopic rebinding time}

To satisfy the second constraint \eqref{eq:nonlocal-constraint2} we need both the microscopic and the mesoscopic mean rebinding times. The microscopic rebinding time is derived in \cite{HHP2}, as
\begin{align}
\label{eq:microrebind}
\microrebind = \frac{L^d}{k_r}.
\end{align}
The mesoscopic rebinding time is simply given by 
\begin{align}
\label{eq:mesorebind1}
\mesorebind = \tau_0,
\end{align}
as $\tau_0$ by definition is the time until an $A$ and a $B$ molecule react, given that they start in the same voxel. We have already derived $\tau_0$ in terms of $\tau_1$ in \eqref{eq:tau1p_tau0}.
\begin{theorem}
Let $\mesorebind$ be the average rebinding time of an $A$ and a $B$ molecule. Then
\begin{align}
\label{eq:mesorebind_theorem}
\mesorebindr \approx t_e^0+\frac{p_0}{p_0+2dr\timejump }\frac{N}{\krme}.
\end{align}
\end{theorem}
\begin{proof}
This follows immediately from Lemma~\ref{lemma-tau1} and \eqref{eq:tau1p_tau0}.
\end{proof}

\subsubsection{Solving for $r$ and $\krme$}

We now want $r$ and $\krme$ to satisfy the constraints \eqref{eq:nonlocal-constraint1} and \eqref{eq:nonlocal-constraint2}. It will prove useful to divide the problem into two cases:
\begin{align}
&\text{Case 1: } h\geq\hstarinf\\
&\text{Case 2: } \hstarinf>h\geq \hstarinfnl.
\end{align}
It turns out that in case 1 we get $r=0$, effectively reducing the generalized algorithm to the standard algorithm.
\begin{theorem}
Assume that $r$ and $\krme$ have been chosen to satisfy the first constraint \eqref{eq:nonlocal-constraint1}. Then, for $h\geq\hstarinf$, we have
\begin{align}
\mesorebindr(\krme,h)\gtrapprox\mesorebindlocal(\krmi,h)
\end{align}
\end{theorem}
\noindent Since $\mesorebindlocal(\krmi,h) \geq\microrebind(\krmi)$, it immediately follows that for $h\geq\hstarinf$, the generalized RDME and the standard RDME agree.
\begin{proof}
We already know that
\begin{align}
\label{eq:binding-times}
&\mesobindloc = \step{0} \timejump +\tauloc{0}\\
&\mesobindnl = \step{1} \timejump +\taunl{1}
\end{align}
where $\tau_i$, as previously defined, is the average time until the molecules react, given that the $B$ molecule is in $d_i$. The superscript $g$ indicates that it is the average time in the case of the generalized RDME, and omission of the superscript indicates that it is the average time in the case of the standard RDME. 

We have assumed that \eqref{eq:nonlocal-constraint1} is satisfied and consequently
\begin{align}
\label{eq:hgh-eq1}
0 = \mesobindloc-\mesobindnl = (\step{0} -\step{1} )\timejump +(\tauloc{0}-\taunl{1}) = N\timejump +(\tauloc{0}-\taunl{1}),
\end{align}
where the second equality follows from \eqref{eq:ns1} and \eqref{eq:ns0}. We know that $\mesorebindloc=\tauloc{0}$, so we get
\begin{align}
\mesorebindloc = \taunl{1}-N\timejump .
\end{align}
Thus
\begin{align}
&\mesorebindnl\geq\mesorebindloc\label{eq:ineq1}\\
 \iff &\taunl{0}\geq \taunl{1}-N\timejump  \\
 \iff &\taunl{1}-\taunl{0}\leq N\timejump .\label{eq:hgh-eq2}
\end{align}
We have already shown that
\begin{align}
\label{eq:taus}
&\taunl{0} = t_e^0+\frac{p_0}{p_0+\timejump }\taunl{1}>\frac{p_0}{p_0+\timejump }\taunl{1}\\
&\taunl{1} \approx \frac{p_0+\timejump }{p_0+2dr\timejump }\frac{N}{\krmenl}.
\end{align}
Now \eqref{eq:hgh-eq2} becomes
\begin{align}
\frac{p_0+\timejump }{p_0+2dr\timejump }\frac{N}{\krmenl}-\frac{p_0}{p_0+\timejump }\frac{p_0+\timejump }{p_0+2dr\timejump }\frac{N}{\krmenl}\leq N\timejump \\
\iff \left(\frac{p_0+\timejump }{p_0+2dr\timejump }-\frac{p_0}{p_0+2dr\timejump }\right)\frac{N}{\krmenl} \leq N\timejump \\
\iff \frac{\timejump }{p_0+2dr\timejump }\frac{N}{\krmenl}\leq N\timejump \\
\iff \frac{1}{p_0+2dr\timejump }\frac{1}{\krmenl}\leq 1.\label{eq:hgh-eq3}
\end{align}
By definition, $p_0 = 1/(1-2dr)\krmenl$, so \eqref{eq:hgh-eq3} becomes
\begin{align}
\frac{1}{\frac{1}{(1-2dr)\krmenl}+2dr\timejump }\frac{1}{\krmenl}\leq 1\\
\iff \frac{1}{(1-2dr)^{-1}+2dr\timejump \krmenl}\leq 1\\
\iff 1\leq (1-2dr)^{-1}+2dr\timejump \krmenl.\label{eq:hgh-eq4}
\end{align}
Since $1-2dr\leq1$, we have $(1-2dr)^{-1}\geq 1$, and $2dr\timejump \krmenl\geq 0$ so \eqref{eq:hgh-eq4} is satisfied for all $r$ and $\krmenl$. Thus \eqref{eq:ineq1} holds for all $r$ and $\krmenl$.
\end{proof}
What remains is to determine $r$ and $\krmenl$ for $h<\hstarinf$.
\begin{theorem}
\label{fund_bound_th}
Assume that $h=\hstarinfnl$ and that $\tau_1\gg t_e^0$. Then 
\begin{align}
&\mesobindnl(\krme,h) \approx \microbind(\krmi)\label{eq:smhs-0}\\
&\mesorebindnl(\krme,h) \approx \microrebind(\krmi),\label{eq:smhs-1}
\end{align}
for $\krme = \krmi/\voxsize^d$ and $1-2dr = 0$.

For $h<\hstarinfnl$,
\begin{align}
\label{eq:smhs-2}
\mesorebindnl(\krme,h)\lessapprox\microrebind(\krmi).
\end{align}
\end{theorem}
Note that we have already shown that we can satisfy \eqref{eq:nonlocal-constraint1} at least down to $h=\hstarinfnl$. The assumption $\tau_1\gg t_e^0$ means, in words, that the \emph{average} time until two molecules react, given that they are one voxel apart, is much longer than the average time until the first event, given that they occupy the same voxel. Unless the microscopic reaction rate is very high, this should be a reasonable assumption for most systems. The necessity of this assumption is realized by considering two molecules in the same voxel. Now, if the average microscopic rebinding time is smaller than the average time until the first diffusion event on the mesoscopic scale, we could not hope to find mesoscopic rates that will yield a match between the mesoscopic rebinding time and the microscopic rebinding time. 
\begin{proof}
We already know that
\begin{align}
\mesobindnl(\krme,h) = \step{1}\timejump+\tau_1,
\end{align}
and from assuming $\voxsize = \hstarinfnl$, it follows that
\begin{align}
\step{1}\timejump = \microbind(\infty),
\end{align}
and thus
\begin{align}
\mesobindnl(\krme,\hstarinfnl) = \microbind(\infty)+\tau_1.
\end{align}
Lemma 2 yields, for $1-2dr=0$ and $\krme = \krmi/\voxsize^d$,
\begin{align}
\tau_1 = \frac{N}{\krme} = \frac{N\voxsize^d}{\krmi} = \frac{L^d}{\krmi} = \microrebind(\krmi).
\end{align}
We now have
\begin{align}
\mesobindnl\left(\frac{\krmi}{\voxsize^d},\hstarinfnl\right) = \microbind(\infty)+\microrebind(\krmi) = \microbind(\krmi),
\end{align}
and thus \eqref{eq:smhs-0} holds. Since we have assumed $\tau_1\gg t_e^0$ and $1-2dr=0$, we get
\begin{align}
\mesorebindnl\left(\frac{\krmi}{\voxsize^d},\hstarinfnl\right) = \tau_0\approx \tau_1 = \microrebind(\krmi),
\end{align}
and we have shown that \eqref{eq:smhs-1} holds.

It remains to show \eqref{eq:smhs-2}. To this end, we simply note that
\begin{align}
\microbind(\krmi) = \mesobindnl = \step{1} \timejump +\tau_1>\microbind(\infty)+\tau_1,
\end{align}
since $\step{1} \timejump >\microbind(\infty)$ for $h<\hstarinfnl$. Thus
\begin{align}
\microrebind = \microbind(\krmi)-\microbind(\infty) >\tau_1 \approx \tau_0 = \mesorebindnl.
\end{align}
which proves \eqref{eq:smhs-2}.
\end{proof}

We now show that for $\hstarinfnl<h<\hstarinf$ we can match both the mean binding time and the mean rebinding time.

\begin{theorem}
Assume that $\hstarinfnl<h<\hstarinf$ and that $\tau_1\gg t_e^0$. Then we have
\begin{align}
&\mesobindr(\krme,h) \approx \microbind(\krmi)\label{eq:ap-s1}\\
&\mesorebindr(\krme,h) \approx \microrebind(\krmi),\label{eq:ap-s2}
\end{align}
for
\begin{align}
&\krmenl = \left(\frac{\timejump Q^2+\krmi/\voxsize^d}{\timejump Q^2+Q\krmi/\voxsize^d}\right)\frac{\krmi}{h^d}\label{eq:ap-params1}\\
&r = \frac{DQ(Q-1)}{2dDQ^2+\krmi/h^{d-2}},\label{eq:ap-params2}
\end{align}
where
\begin{align}
\label{eq:ap-Q}
Q = \frac{N\timejump }{\microbind(\infty)-\mesobindnl(\infty)} = \begin{cases}
\left[ \frac{2}{\pi}\log\left(\frac{h}{\hstarinfnl}\right)\right]^{-1}\quad (2D)\\
\left[ (C_3-1)\left(\frac{h}{\hstarinfnl}-1\right)\right]^{-1}\quad (3D).
\end{cases}
\end{align}
\end{theorem}
\begin{proof}
We already know that
\begin{align}
&\mesobindnl(\krme,\voxsize) = \step{1} \timejump +\tau_1\label{eq:ap-p1}\\
&\microrebind(\krmi) = \microbind(\krmi)-\microbind(\infty) = \frac{Nh^d}{\krmi}\label{eq:ap-p2}\\
&\mesorebindnl(\krme,\voxsize) = \tau_0 \approx \frac{p_0}{p_0+2dr\timejump }\frac{N}{\krmenl}\label{eq:ap-p3}.
\end{align}
Consequently we satisfy \eqref{eq:ap-s1} if
\begin{align}
\step{1} \timejump +\tau_1 = \microbind(\krmi),
\end{align}
which, by \eqref{eq:tau1-th1} and \eqref{eq:tau1-th2}, approximately holds if
\begin{align}
\label{eq:ap-p4}
\frac{p_0+\timejump }{p_0+2dr\timejump }\frac{N}{\krmenl} = \microbind(\krmi)-\step{1} \timejump .
\end{align}
To satisfy \eqref{eq:ap-s2}, we must, by \eqref{eq:ap-p2} and \eqref{eq:ap-p3}, satisfy
\begin{align}
\label{eq:ap-p5}
\frac{p_0}{p_0+2dr\timejump }\frac{N}{\krmenl} = \microrebind(\krmi).
\end{align}
Subtracting both the right-hand and left-hand side of \eqref{eq:ap-p5} from \eqref{eq:ap-p4},  we obtain
\begin{align}
\label{eq:ap-p6}
\frac{\timejump }{p_0+2dr\timejump }\frac{N}{\krmenl} = \microbind(\krmi)-\step{1} \timejump -\microrebind(\krmi).
\end{align} 
By definition, $p_0=1/(1-2dr)\krmenl$ and $\timejump=\voxsize^2/(2dD)$, so \eqref{eq:ap-p6} yields, after some straightforward algebra,
\begin{align}
\label{eq:ap-p7}
\krmenl = \left( \frac{(1-2dr)N\timejump }{\microbind(\krmi)-\step{1} \timejump -\microrebind(\krmi)}-1 \right)\frac{D}{rh^2(1-2dr)}.
\end{align}
Since $\mesobindnl(\infty) = \step{1} \timejump $ and $\microrebind(\krmi) = \microbind(\krmi)-\microbind(\infty)$, \eqref{eq:ap-p7} becomes
\begin{align}
\krmenl &= \frac{D}{rh^2}\left(\frac{N\timejump }{\microbind(\infty)-\mesobindnl(\infty)}-\frac{1}{1-2dr}\right)\label{eq:ap-p8-1}\\
&=\frac{D}{rh^2}\left( Q-\frac{1}{1-2dr}\right).\label{eq:ap-p8-2}
\end{align}
With $\krmenl$ as in \eqref{eq:ap-p8-2}, we want to find $r$ such that \eqref{eq:ap-p5} is satisfied. Since $\microrebind=L^d/\krmi=Nh^d/\krmi$, we obtain
\begin{align}
&\frac{p_0}{p_0+2dr\timejump }\frac{N}{\krmenl} = \frac{Nh^d}{\krmi}\\
\iff &\frac{1}{1_2dr\timejump p_0^{-1}}\frac{1}{\krmenl} = \frac{h^d}{\krmi}\\
\iff &(1+2dr\timejump p_0^{-1})\krmenl = \frac{\krmi}{h^d}.\label{eq:ap-p9}
\end{align}
Since $\timejump  = h^2/2dD$ and $p_0=1/(1-2dr)\krmi$, \eqref{eq:ap-p9} becomes
\begin{align}
\label{eq:ap-p10}
\frac{rh^2}{D}(1-2dr)(\krmenl)^2+\krmenl = \frac{\krmi}{h^d}.
\end{align}
We expand the first term of the left-hand side to get
\begin{align}
\frac{rh^2}{D}(1-2dr)(\krmenl)^2 = \frac{D}{rh^2}\left[ (1-2dr)Q^2-2Q+\frac{1}{1-2dr} \right].
\end{align}
Thus
\begin{align}
\frac{rh^2}{D}(1-2dr)(\krmenl)^2+\krmenl = \frac{D}{rh^2}\left[ (1-2dr)Q^2-Q\right],
\end{align}
and \eqref{eq:ap-p10} becomes
\begin{align}
\frac{D}{rh^2}\left[ (1-2dr)Q^2-Q\right] = \frac{\krmi}{h^d},
\end{align}
yielding
\begin{align}
r = \frac{DQ(Q-1)}{2dDQ^2+\frac{\krmi}{h^d}}.
\end{align}
Inserting $r$ above into \eqref{eq:ap-p8-2} yields \eqref{eq:ap-params1}. 

It remains to show that $\krmenl>0$ and $0<r<1/(2d)$ hold for $\krmenl$ and $r$ given by \eqref{eq:ap-params1} and \eqref{eq:ap-params2}. We first show that $Q>1$, from which $r>0$ follows. Thus we should show that
\begin{align}
\label{eq:ap-p11}
Q=\frac{N\timejump }{\microbind(\infty)-\mesobindnl(\infty)}>1
\end{align}
holds. We start by showing that \eqref{eq:ap-p11} holds in 3D. By \eqref{eq:ns1},
\begin{align}
\mesobindnl(\infty,h) \approx (C_3-1)N\timejump  = (C_3-1)\frac{L^3}{h^3}\frac{h^2}{2dD},
\end{align}
for $N\gg 1$. We have already shown that
\begin{align}
\microbind(\infty) = \mesobindr(\infty,\hstarinfnl) = (C_3-1)\frac{L^3}{(\hstarinfnl)^3}\timejump ,
\end{align}
so \eqref{eq:ap-p11} becomes
\begin{align}
Q=&\frac{\frac{L^3}{h^3}\frac{h^2}{2dD}}{(C_3-1)\frac{L^3}{(\hstarinfnl)^3}\frac{(\hstarinfnl)^2}{2dD}-(C_3-1)\frac{L^3}{h^3}\frac{h^2}{2dD}}
= \frac{\frac{1}{h}}{(C_3-1)\left(\frac{1}{\hstarinfnl}-\frac{1}{h}\right)}>1\\
\iff &(C_3-1)\left( \frac{h}{\hstarinfnl}-1\right)<1.
\end{align}
Since $\hstarinf/\hstarinfnl = C_3/(C_3-1)$, and, by assumption, $h<\hstarinf$, we obtain
\begin{align}
(C_3-1)\left(\frac{h}{\hstarinfnl}-1\right)<(C_3-1)\left(\frac{C_3}{C_3-1}-1\right)=1.
\end{align}
Thus $Q>1$, and, as a consequence, $r>0$. In 2D we have
\begin{align}
\mesobindnl(\infty,h) &= \left[\pi^{-1}N\log N+(C_2-1)N\right]\timejump \\
 &= \left[\pi^{-1}\frac{L^2}{h^2}\log \frac{L^2}{h^2}+(C_2-1)\frac{L^2}{h^2}\right]\frac{h^2}{2dD} \\
 &= \pi^{-1}\frac{L^2}{2dD}\log\frac{L^2}{h^2}+(C_2-1)\frac{L^2}{2dD}.
\end{align}
In 2D, similarly as in the 3D case, we have $\microbind(\infty) = \mesobindnl(\infty,\hstarinfnl)$. Thus
\begin{align}
Q &= \frac{\frac{L^2}{h^2}\frac{h^2}{2dD}}{\left[ \pi^{-1}\frac{L^2}{2dD}\log\frac{L^2}{(\hstarinfnl)^2}+(C_2-1)\frac{L^2}{2dD}\right]-\left[\pi^{-1}\frac{L^2}{2dD}\log\frac{L^2}{h^2}+(C_2-1)\frac{L^2}{2dD}\right]} \\
&= \frac{1}{2\pi^{-1}\left(\log\frac{L}{\hstarinfnl}-\log\frac{L}{h}\right)} = \frac{1}{2\pi^{-1}\log\frac{h}{\hstarinfnl}}.
\end{align}
Since, by assumption,
\begin{align}
1<\frac{h}{\hstarinfnl}<\frac{\hstarinf}{\hstarinfnl},
\end{align}
and
\begin{align}
\frac{\hstarinf}{\hstarinfnl} = \exp\left[\frac{3+2\pi C_2}{4}-\frac{3+2\pi(C_2-1)}{4}\right] = \exp\left(\frac{\pi}{2}\right),
\end{align}
we obtain
\begin{align}
Q = \frac{1}{2\pi^{-1}\log\frac{h}{\hstarinfnl}}>\frac{1}{2\pi^{-1}\log\left(\exp\frac{\pi}{2}\right)} = 1.
\end{align}
Thus $Q>1$ holds in both 2D and 3D, and we have $r>0$. Note that with $Q>1$, $\krmenl>0$ follows immediately. It remains to show $r<1/(2d)$. To this end, we simply note that
\begin{align}
r = \frac{DQ^2-DQ}{2dDQ^2+\frac{\krmi}{h^{d-2}}} = \left(\frac{DQ^2-DQ}{DQ^2+\frac{\krmi}{2dh^{d-2}}}\right)\frac{1}{2d},
\end{align}
where
\begin{align}
\frac{DQ^2-DQ}{DQ^2+\frac{\krmi}{2dh^{d-2}}}<1
\end{align}
holds, since $Q>1$.

Thus $0<r<1/(2d)$ and $\krmenl>0$ for $\hstarinfnl<h<\hstarinf$.
\end{proof}

\subsection{Dissociation rates}
\label{sec:dissociation}

Consider the same setup as before, with one $A$ molecule and one $B$ molecule reacting reversibly according to $\revreaction{A}{B}{\krmi}{\kdmi}{C}$. Above we have determined how to choose the mesoscopic association rates, so what remains is to determine the dissociation rate. This can be done completely analogously to the case of the standard RDME. We thus follow the approach of \cite{HHP2}, and conclude that we must have
\begin{align}
\frac{(\kdme)^{-1}}{\mesorebindr+(\kdme)^{-1}} = \frac{\kdmi^{-1}}{\microrebind+\kdmi^{-1}},
\end{align}
to obtain a steady state on the mesoscopic scale that matches the steady state of the microscopic scale. Thus it follows immediately that for $\hstarinfnl \leq h \leq \hstarinf$, we should have
\begin{align}
\kdme = \kdmi,
\end{align}
because $\mesorebindr(\krme,h) = \microrebind(\krmi)$ holds.
\subsection{Summary of the algorithm}
\label{sec:algsummary}

Assume that we have a cubic (3D) or square (2D) domain of width $L$, discretized by a Cartesian mesh with voxels of width $\voxsize$. Consider a reversible reaction $\revreaction{A}{B}{\krmi}{\kdmi}{C}$, where $\krmi$ and $\kdmi$ are the microscopic reaction rates. Let $D=D_A+D_B$, where $D_A$ and $D_B$ are the dissociation rates of species $A$ and $B$, respectively. Let $\sigma = \sigma_A+\sigma_B$ be the reaction radius of an $A$ and a $B$ molecule.

The critical mesh sizes are given by
\begin{align}
\hstarinf \approx \begin{cases}
5.1\sigma\quad (2D)\\
3.2\sigma\quad (3D),
\end{cases}
\end{align}
for the standard RDME, and the critical mesh sizes for the generalized RDME are given by
\begin{align}
\hstarinfnl \approx \begin{cases}
1.06\sigma\quad (2D)\\
1.08\sigma\quad (3D),
\end{cases}
\end{align}

We now wish to simulate this system on the mesoscopic scale with the generalized RDME. The results of this section can be summarized as follows:
\begin{itemize}
\item For $h\geq\hstarinf$: The generalized RDME reduces to the standard RDME. Thus $r=0$ and $\krme = \localrated(\krmi,\voxsize)$. Molecules react only when occupying the same voxel. The dissociation rate is given by $\kdme = \voxsize^d\kdmi\krme/\krmi$, as shown in \cite{HHP2}.
\item For $\hstarinfnl<h<\hstarinf$: We match both the mean binding time, and the mean rebinding time, of the $A$ and $B$ molecules by choosing $r$ and $\krme$ as in \eqref{eq:ap-params1} and \eqref{eq:ap-params2}. Now molecules react with an intensity of $r\krme$ when occupying neighboring voxels, and with an intensity of $(1-2dr)\krme$ when occupying the same voxel. The dissociation rate is simply given by $\kdme=\krmi$.
\item For $h<\hstarinfnl$ we can no longer match the mean rebinding time, and the accuracy deteriorates with decreasing $h$. 
\end{itemize}

\section{Numerical results}
\label{sec:results}

\subsection{Rebinding-time distributions}

Consider a system of two species $A$ and $B$, with one molecule of each. The $A$ molecule is fixed at the origin, while the $B$ molecule diffuses freely in space. In \cite{HHP2} it was shown that the local RDME matched the microscopic rebinding-time distribution for a reversibly reacting pair down to $t^*\sim (\hstarinf)^2/(2\diffconst)$. For $t<t^*$, the behavior is inevitably going to be different, as the accuracy of the RDME is inherently limited by the spatial resolution.

With the generalized RDME, we can match both the average binding times as well as the average rebinding times for $\hstarinf\geq h\geq \hstarinfnl$, and thus we could hope that also the error in distribution will be small at timescales of $(\hstarinfnl)^2/(2D)<t<(\hstarinf)^2/(2D)$.

In Fig. \ref{fig:rebind_distribution} we compare the microscopic rebinding-time distribution to the rebinding-time distribution for the generalized RDME. As we can see, there is a good match down to a spatial resolution of approximately $\sigma$. For the finest meshes, the behavior at really short time scales is incorrect due to dissociating particles starting in the same voxel, but not reacting until they are in neighboring voxels. This introduces an error on the order of the voxel size, which will be on the order of the size of the molecules. 

\begin{figure}[htp]
\centering
\subfigure{\includegraphics[width=0.49\linewidth]{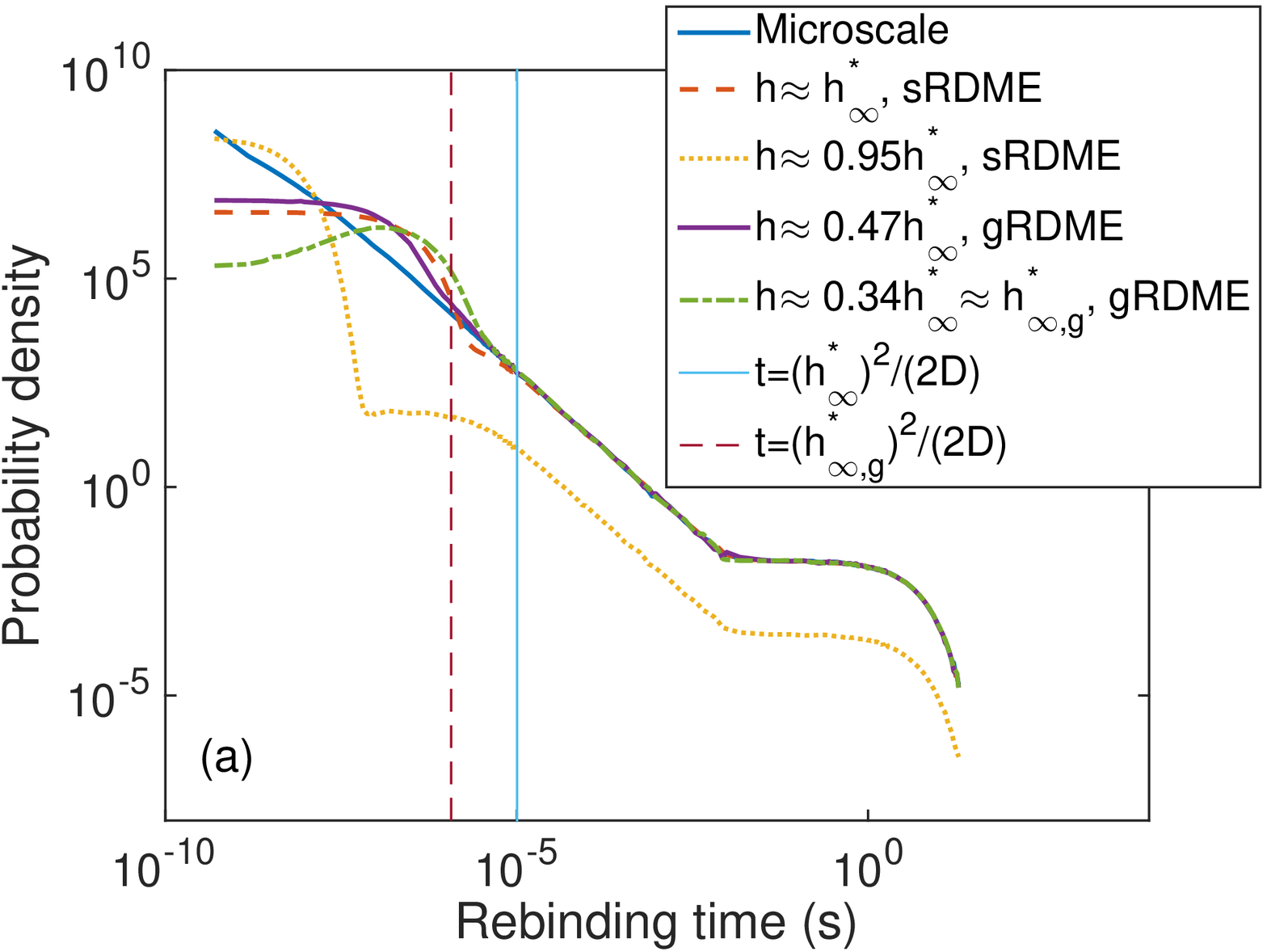}}
\subfigure{\includegraphics[width=0.49\linewidth]{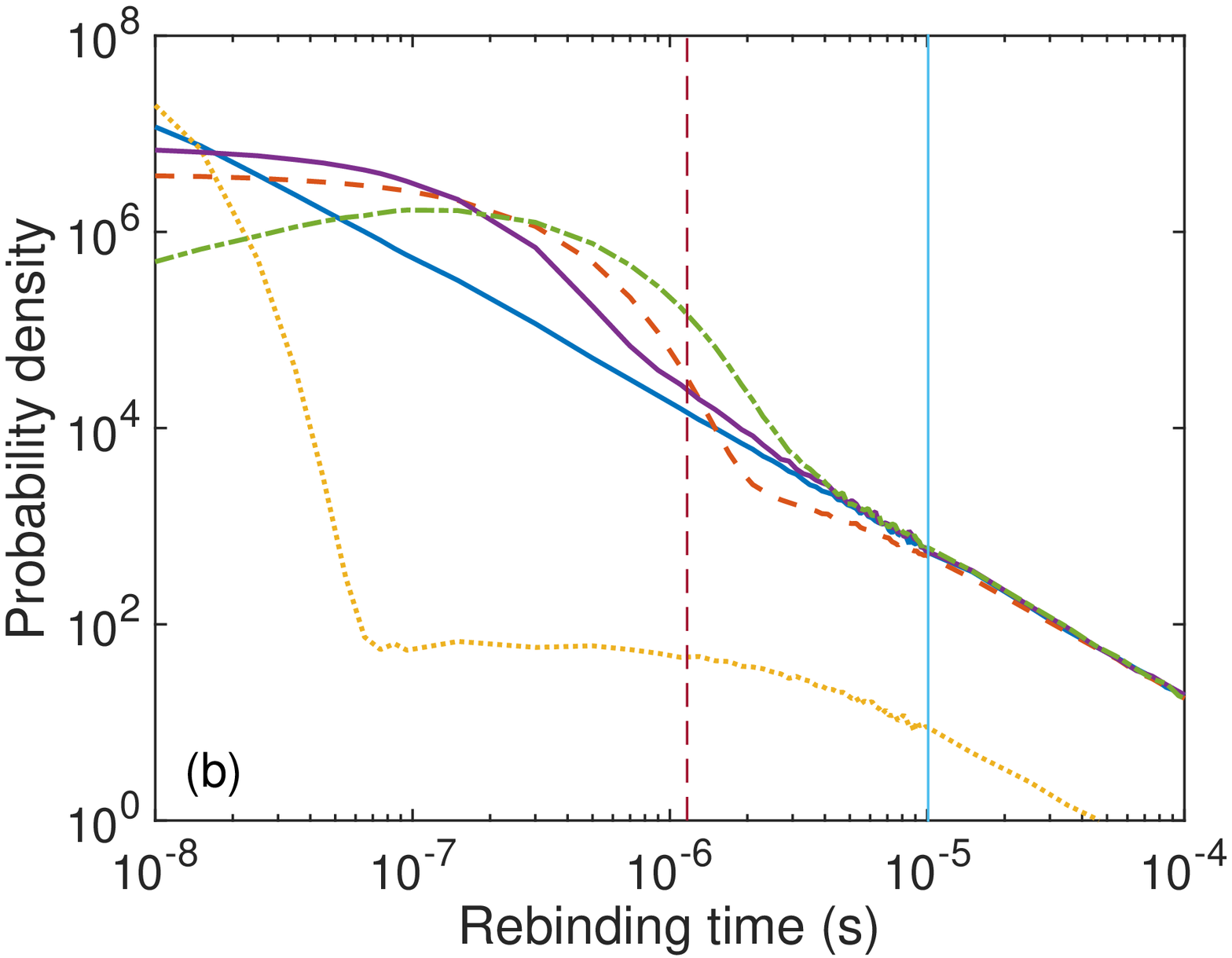}}
\subfigure{\includegraphics[width=0.49\linewidth]{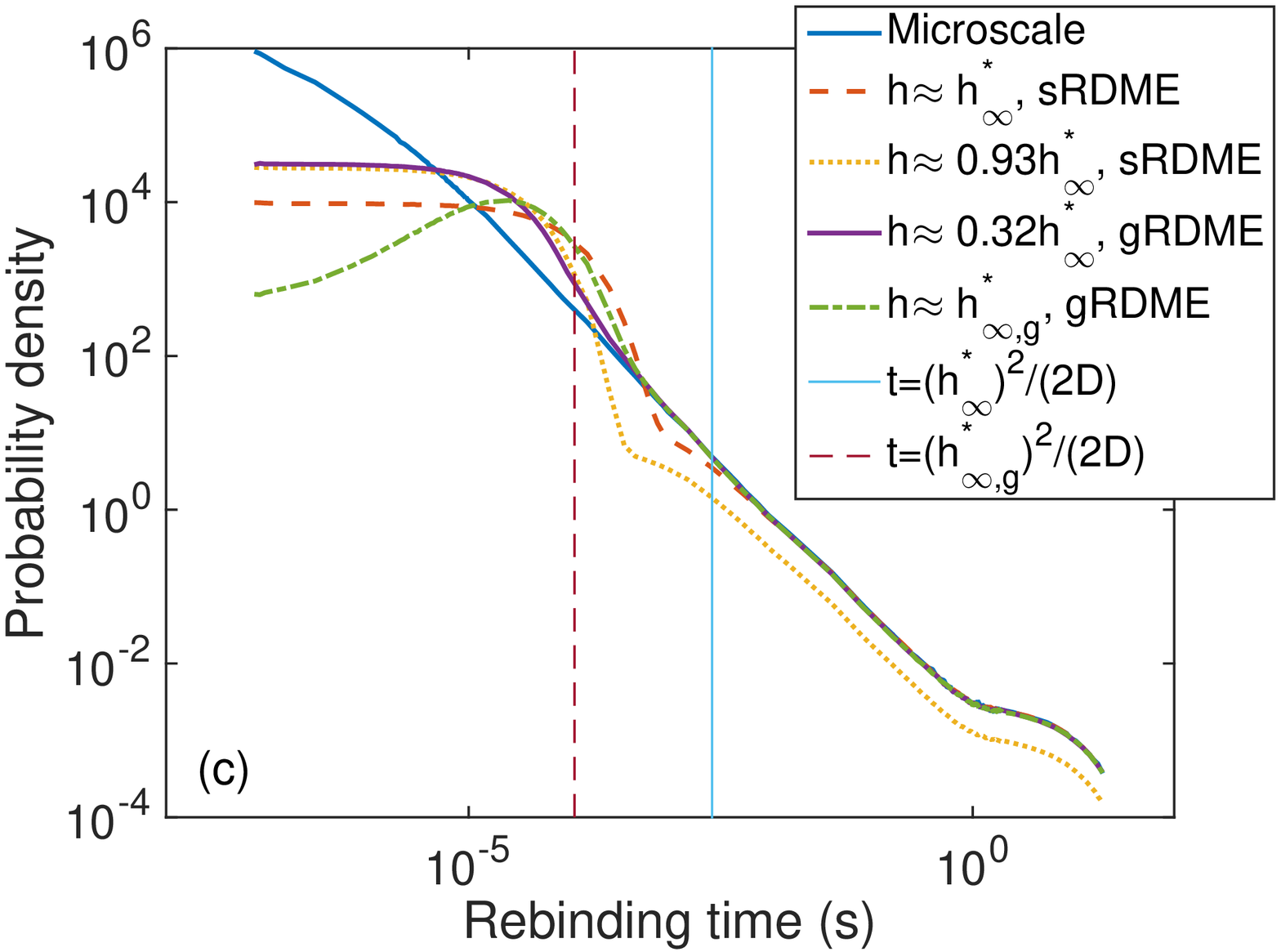}}
\subfigure{\includegraphics[width=0.49\linewidth]{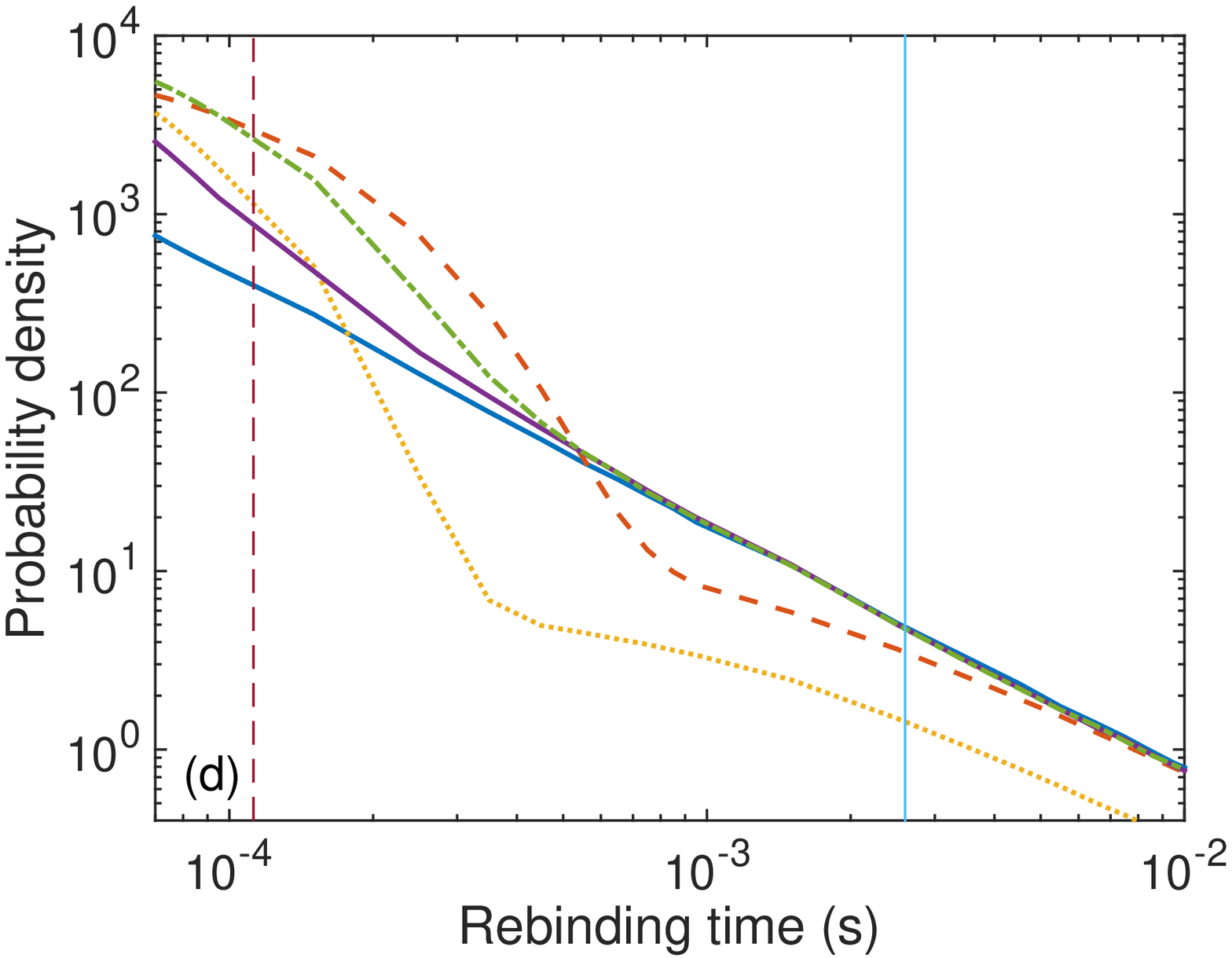}}
\caption{\label{fig:rebind_distribution}In (a) and (b) we have plotted the rebinding-time distributions in 3D. For the standard RDME we have a good match between the microscopic and mesoscopic simulations for $h\approx\hstarinf$, while the average rebinding time is underestimated for finer meshes. For the generalized RDME we see that the microscopic and mesoscopic distributions agree well for $\hstarinfnl<h<\hstarinf$ down to spatial resolution almost on the order of the size of the molecules, or a temporal resolution of approximately $(\hstarinfnl)^2/(2D)$. In (c) and (d) we have plotted the rebinding-time distributions in 2D. The conclusions are the same as for the 3D case. The parameters in (a) and (b) are given by $\sigma=\SI{2d-9}{\m}$, $D=\SI{2d-12}{\m\squared\per\s}$, $L=\SI{5.145d-7}{\m}$, and $\krmi=\SI{d-18}{\m\cubed\per\second}$. The parameters in (c) and (d) are given by $\rradius=\SI{2d-9}{\m}$, $D=\SI{2d-14}{\m\squared\per\s}$, $L=\SI{5.2d-7}{\m}$, and $\krmi=\SI{d-12}{\m\squared\per\s}$. Note that `standard RDME' has been shortened to sRDME, and `generalized RDME' to gRDME in the legends, to increase legibility.}
\end{figure}

\subsection{Convergence of the generalized RDME}

The dynamics of some systems is resolved only at a fine spatial resolution. In particular, it has been shown that fast rebinding events can affect e.g. the response time of a MAPK pathway \cite{TaTNWo10}. We consider the system
\begin{align}
\label{eq:ex3-system}
\begin{cases}
S_1 &\overset{k_d}{\rightarrow} S_{11}+S_{12} \overset{\krmi}{\rightarrow} S_2\\
S_2 &\overset{k_d}{\rightarrow} S_{21}+S_{22} \overset{\krmi}{\rightarrow} S_3,
\end{cases}
\end{align}
which has a behavior similar to the MAPK pathway of \cite{TaTNWo10}. Due to the possibility of fast rebinding events, the long-term dynamics of the system is affected by spatial correlations between newly produced molecules.

We start with an initial population of 100 $S_1$ molecules, with none of the other species present. The system is simulated for 2s, during which we sample the state of the system at 201 evenly distributed points between $t=0$ and $t=2$. We simulate the system with both the standard RDME, as well as with the generalized extension, for different voxel sizes. Let $\mathcal{S} =\{ S_1,S_{11},S_{12},S_{2},S_{21},S_{22},S_{3} \}$ . The error is then computed as
\begin{align}
\label{eq:ex3-error}
E(h) = \frac{1}{201}\sum_{i=1}^{201} \sum_{S\in\cal{S}}\left|[S]_{h,i}^{meso}-[S]_i^{micro}\right|,
\end{align}
where $[S]_i^{micro}$ is the average population of $S$ at time $t_i$, obtained with the microscopic GFRD algorithm, and where $[S]_{h,i}^{meso}$ denotes the average population of $S$ at time $t_i$ obtained at the mesoscopic scale with voxel size $h$.

After a dissociation of either an $S_1$ or $S_2$ molecule from \eqref{eq:ex3-system}, the products can rebind quickly to produce an $S_2$ or $S_3$ molecule, respectively. On the microscopic scale, the products are in contact after a dissociation event, and thus the spatial correlation will be significant. At the mesoscopic scale, the products are placed in the same voxel after a dissociation. If the voxel size is large compared to the size of the molecules, the spatial correlation will be less than on the microscale. Thus, to simulate \eqref{eq:ex3-system} accurately, we would expect a fine mesh resolution to be required.

Let $\sigma_i$ be the reaction radius of molecule $S_i$, and $\sigma_{ij}$ the reaction radius of molecule $S_{ij}$. The parameters of the model are given by
\begin{align}
\label{eq:ex3-parameters}
\begin{cases}
k_d = \SI{10}{\per\second}\\
\krmi = \SI{d-19}{\m\cubed\per\s}
\end{cases}
\begin{cases}
\sigma_1 = \SI{d-9}{\m}\\
\sigma_{11} = \sigma_{12} = \SI{0.8d-9}{\m}\\
\sigma_2 = \SI{2d-9}{\m}\\
\sigma_{21}=\sigma_{22} = \SI{1.8d-9}{\m}\\
\sigma_{3} = \SI{2.5d-9}{\m}.
\end{cases}
\end{align}
For simplicity, we let all species have the same diffusion rate, $D=\SI{d-12}{\m\squared\per\s}$. The $S_1$ molecules are initialized uniformly in a cube of volume $\SI{d-18}{\m\cubed}$.

There is a critical lower bound on the mesh size associated with each of the system's bimolecular reaction events
\begin{align}
\label{eq:ex3-lower_bound_h}
\begin{cases}
h_1 := h^{\ast}_{\infty}(\sigma_{11}+\sigma_{12})\approx 5.0815\cdot 10^{-9}\\
h_2 := h^{\ast}_{\infty}(\sigma_{21}+\sigma_{22})\approx 1.1433\cdot 10^{-8}.
\end{cases}
\end{align}
We know that for $\voxsize>\hstarinf$, we are unable to match both the mesoscopic mean association time and the mesoscopic mean rebinding time to the corresponding microscopic quantities. Thus, for $h>\max\{h_1,h_2\}$, we will overestimate the rebinding time for both reactions, and consequently underestimate the average $S_3$ concentration.

For $h_2>h>h_1$ the dynamics is less obvious; we are underestimating the average rebinding time for the first reaction, but overestimating the average rebinding time for the second. As we can see in Fig.~\ref{fig:ex3-fig1} (b), the positive and negative errors partly cancel out in this regime. At first the error decreases with decreasing $h$, but as we approach $h_1$, it starts to increase again. The behavior of the standard RDME is hard to predict, and \emph{a priori} we cannot be sure that a particular choice of $h$ is suitable.

In contrast, we see that the generalized RDME has a more predictable behavior, converging with decreasing $h$, and yielding an almost perfect match for $h<h_1$. The difference in behavior is due to the generalized RDME matching the average rebinding time also for $h<\hstarinf$, all the way down to $h_{2,g}:=\hstarinfnl(\sigma_{21}+\sigma_{22}) \approx 3.8934\cdot 10^{-9}$.

\begin{figure}[htp]
\centering
\subfigure{\includegraphics[width=0.49\linewidth]{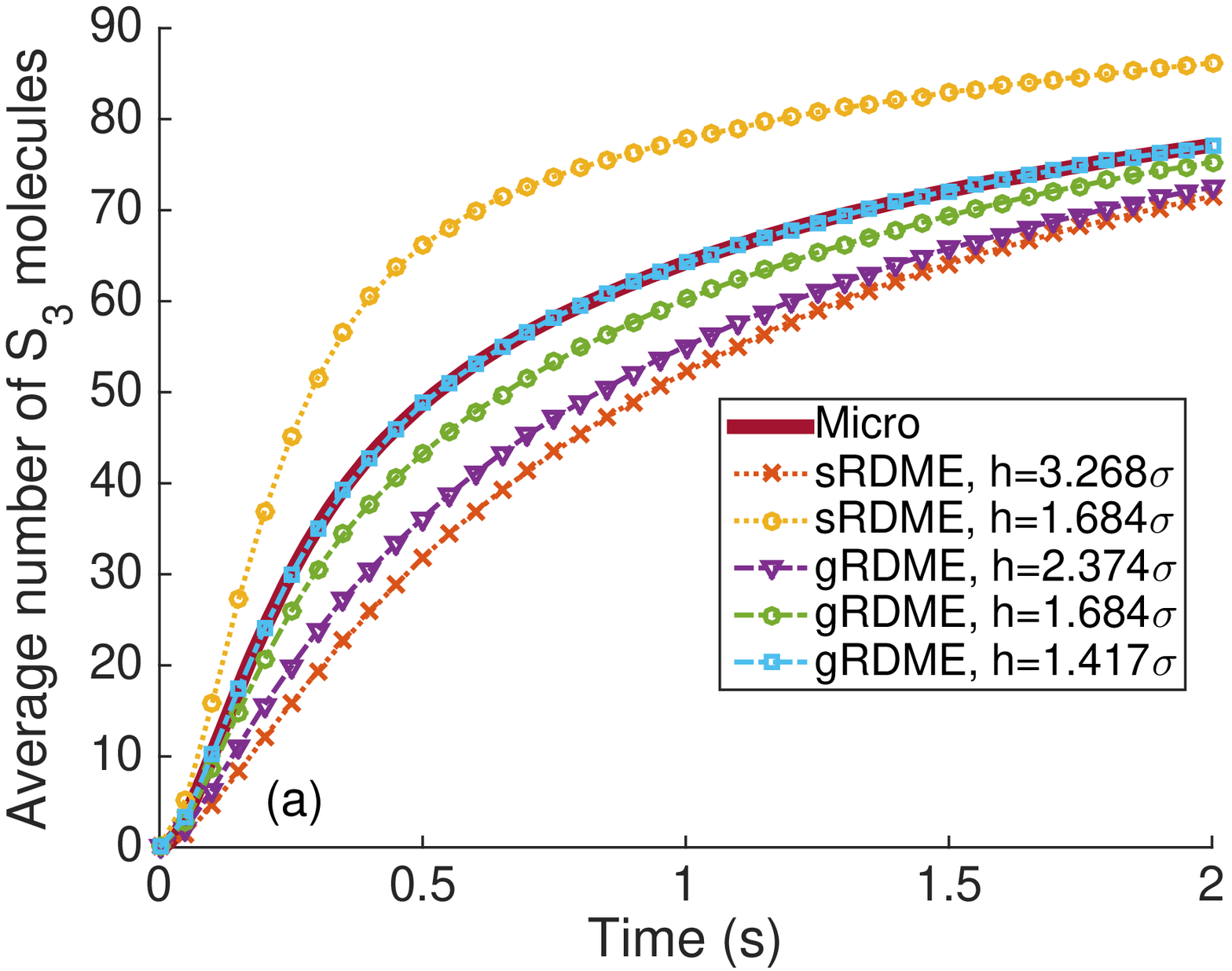}}
\subfigure{\includegraphics[width=0.49\linewidth]{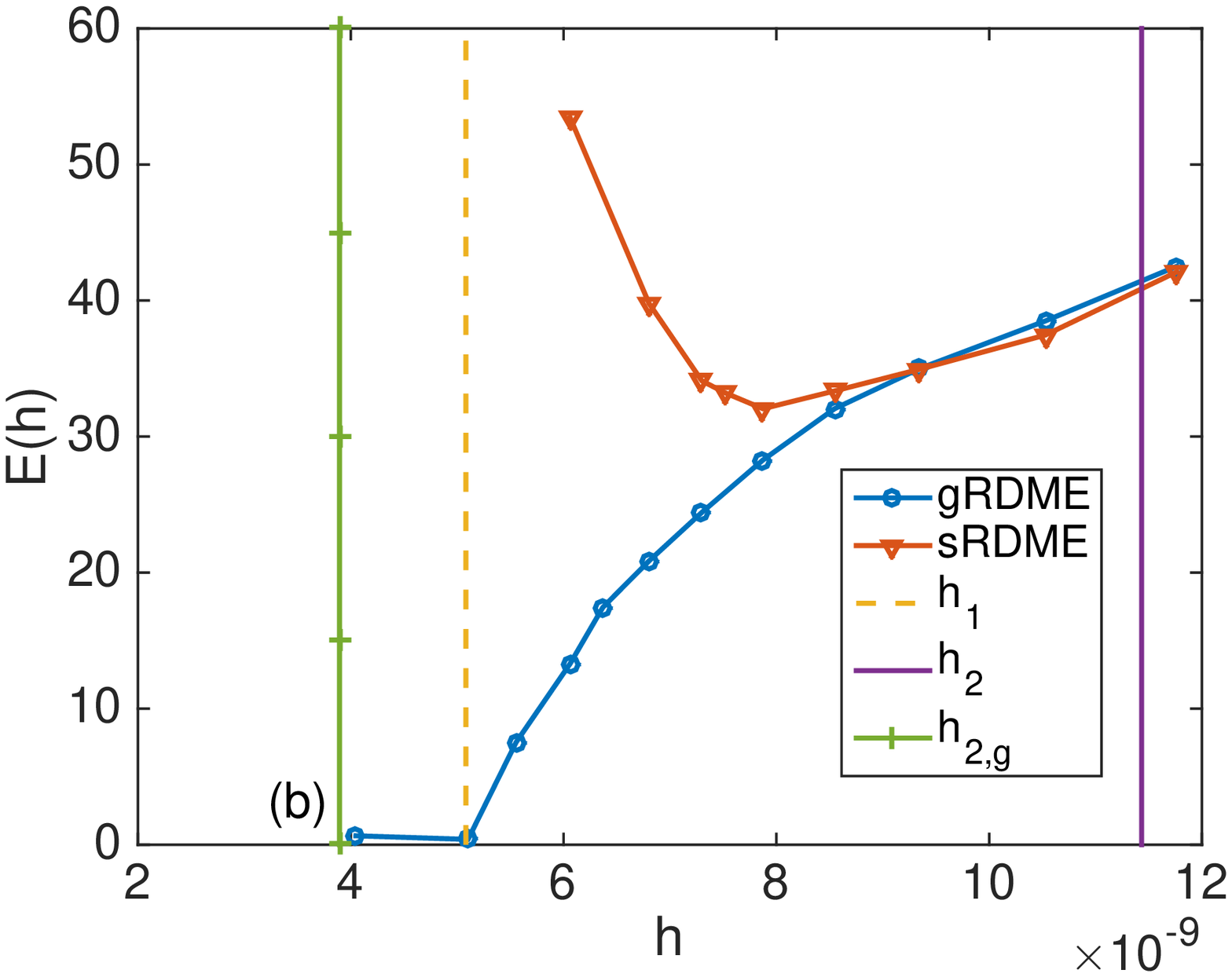}}
\caption{\label{fig:ex3-fig1}In (a) we plot the average number of $S_3$ molecules as a function of time. As we can see, for a larger value of the voxel size $h$, we underestimate the number of $S_3$ molecules. For a very fine mesh, the number of $S_3$ molecules is overestimated. Somewhere in between we may obtain a good approximation compared to the microscopic results, but then the concentration of other species in the system will be incorrect. For simulations with the generalized algorithm, the average number of $S_3$ molecules is underestimated for coarse meshes, but as we refine the mesh, the dynamics approach that of the microscopic simulations. In (b) we see that the total error, as defined in \eqref{eq:ex3-error}, decreases down to $h_1$ for the generalized RDME, while error for the local RDME first decreases slightly but then increases as we refine the mesh further. We obtain an almost perfect match between the microscopic and the generalized RDME as $h$ approaches $h_1$, and all the way down to $h_{2,g}$. As in the legend of Fig. 1, we have shortened `standard RDME' to sRDME, and `generalized RDME' to gRDME in the legends.}
\end{figure}
\section{Summary}

For the standard RDME there is a lower bound on the mesh size, $\hstarinf$, below which the accuracy deteriorates. For $\voxsize>\hstarinf$ we match the mean binding time of two molecules with the mesoscopic reaction rate given by $\localrated(\krmi,h)$. For $\voxsize=\hstarinf$ we match both the mean binding time and the mean rebinding time of the two molecules.

Some systems display fine-grained dynamics, requiring a fine spatial resolution to be simulated at the mesoscopic scale. By generalizing the standard RDME to allow reactions between molecules in neighboring voxels, we obtain a lower bound on the mesh size given by $\hstarinfnl$, where $\hstarinfnl$ is on the order of the reaction radius of a pair of molecules. We derived analytical expressions for the reaction rates, and showed that we match both the mean binding time and the mean rebinding time for $\hstarinfnl\leq\voxsize\leq\hstarinf$. For $\voxsize>\hstarinf$, the generalized RDME and the standard RDME agree.

We studied the accuracy of the generalized RDME in two numerical examples. In the first example we showed that we not only match the mean rebinding time for $\hstarinfnl\leq\voxsize\leq\hstarinf$, but that we also obtain a good match between the rebinding-time distributions at the two scales. In the second example we considered a system that cannot be accurately simulated with the standard RDME, as the mesh resolution required is below the fundamental lower limit $\hstarinf$. We showed that with the generalized RDME we are able to simulate the system to high accuracy, and we showed how we obtain convergence to the microscopic simulations with decreasing mesh size $h$.

\section{Acknowledgments}

This work was funded by NSF Grant No. DMS-1001012, NIGMS of the NIH under Grant No. R01- GM113241-01, Institute of Collaborative Biotechnologies Grant No. W911NF-09-D-0001 from the U.S. Army Research Office, NIBIB of the NIH under Grant No. R01-EB014877-01, and U.S. DOE Grant No. DE-SC0008975. The content of this paper is solely the responsibility of the authors and does not necessarily represent the official views of these agencies.

\newcommand{\noopsort}[1]{}


\begin{thebibliography}{10}

\bibitem{Lawson:2013}
M.~J. Lawson, B.~Drawert, M.~Khammash, and L.~Petzold.
\newblock Spatial stochastic dynamics enable robust cell polarization.
\newblock {\em PLoS Comput. Biol.}, 9(7):e1003139, 2013.

\bibitem{Howard}
M.~Howard and A.~D. Rutenberg.
\newblock Pattern formation inside bacteria: fluctuations due to the low copy
  number of proteins.
\newblock {\em Phys. Rev. Lett.}, 90(12):128102, 2003.

\bibitem{FaEl}
J.~Elf and D.~Fange.
\newblock Noise induced {M}in phenotypes in \textit{{E}. coli}.
\newblock {\em PLoS Comput. Biol.}, 2(6):e80, 2006.

\bibitem{Sturrock1:2013}
M.~Sturrock, A.~Hellander, A.~Marsavinos, and M.~Chaplain.
\newblock Spatial stochastic modeling of the hes1 pathway: Intrinsic noise can
  explain heterogeneity in embryonic stem cell differentiation.
\newblock {\em J. Roy. Soc. Interface}, 10(80):20120988, 2013.

\bibitem{Sturrock2:2013}
M.~Sturrock, A.~Hellander, S.~Aldakheel, L.~Petzold, and M.~Chaplain.
\newblock The role of dimerisation and nuclear transport in the hes1 gene
  regulatory network.
\newblock {\em Bull. Math. Biol.}, 76(4):766--798, 2013.

\bibitem{TaTNWo10}
K.~Takahashi, S.~T{\u a}nase-Nicola, and P.~R. ten Wolde.
\newblock Spatio-temporal correlations can drastically change the response of a
  {MAPK} pathway.
\newblock {\em Proc. Natl. Acad. Sci. USA.}, 107(6):2473--2478, 2010.

\bibitem{ElEh04}
J.~Elf and M.~Ehrenberg.
\newblock Spontaneous separation of bi-stable biochemical systems into spatial
  domains of opposite phases.
\newblock {\em Syst. Biol.}, 1(2):230--236, 2004.

\bibitem{URDME_BMC}
B.~Drawert, S.~Engblom, and A.~Hellander.
\newblock {URDME}: a modular framework for stochastic simulation of
  reaction-transport processes in complex geometries.
\newblock {\em BMC Syst. Biol.}, 6(1):76, 2012.

\bibitem{steps}
I.~Hepburn, W.~Chen, S.~Wils, and E.~De Schutter.
\newblock {STEPS}: efficient simulation of stochastic reaction-diffusion models
  in realistic morphologies.
\newblock {\em BMC Syst. Biol.}, 6(35):425--437, 2012.

\bibitem{mesoRD}
J.~Hattne, D.~Fange, and J.~Elf.
\newblock Stochastic reaction-diffusion simulation with {M}eso{RD}.
\newblock {\em Bioinformatics}, 21(12):2923--2924, 2005.

\bibitem{Tomita99}
M.~Tomita, K.~Hashimoto, K.~Takahashi, T.~S. Shimizu, Y.~Matsuzaki, F.~Miyoshi,
  K.~Saito, S.~Tanida, K.~Yugi, J.~C. Venter, and C.~A. Hutchison.
\newblock E-cell: software environment for whole-cell simulation.
\newblock {\em Bioinformatics}, 15(1):72--84, 1999.

\bibitem{Smol}
M.~v.~Smoluchowski.
\newblock Versuch einer mathematischen {T}heorie der {K}oagulationskinetik
  kolloider {L}{\"o}sungen.
\newblock {\em Z.~phys.~Chemie}, 92(2):129--168, 1917.

\bibitem{AnAdBrAr10}
S.~S. Andrews, N.~J. Addy, R.~Brent, and A.~P. Arkin.
\newblock Detailed simulations of cell biology with {S}moldyn 2.1.
\newblock {\em PLoS Comput. Biol.}, 6(3):e1000705, 2010.

\bibitem{MCell08}
Rex~A. Kerr, Thomas~M. Bartol, Boris Kaminsky, Markus Dittrich, Jen-Chien~Jack
  Chang, Scott~B. Baden, Terrence~J. Sejnowski, and Joel~R. Stiles.
\newblock Fast {M}onte {C}arlo simulation methods for biological
  reaction-diffusion systems in solution and on surfaces.
\newblock {\em SIAM J. Sci. Comput.}, 30(6):3126--3149, 2008.

\bibitem{HHP}
S.~Hellander, A.~Hellander, and L.~R. Petzold.
\newblock Reaction-diffusion master equation in the microscopic limit.
\newblock {\em Phys. Rev. E}, 85(4):042901, 2012.

\bibitem{HHP2}
S.~Hellander, A.~Hellander, and L.~R. Petzold.
\newblock Reaction rates for mesoscopic reaction-diffusion kinetics.
\newblock {\em Phys. Rev. E}, 91(2):023312, 2015.

\bibitem{IsaacsonCRDME}
S.~Isaacson.
\newblock A convergent reaction-diffusion master equation.
\newblock {\em J. Chem. Phys.}, 139(5):054101, 2013.

\bibitem{FangeSRDME}
D.~Fange, O.~G. Berg, P.~Sj{\"o}berg, and J.~Elf.
\newblock Stochastic reaction-diffusion kinetics in the microscopic limit.
\newblock {\em PNAS}, 107(46):19820--Ð19825, 2010.

\bibitem{Doi1}
M.~Doi.
\newblock Second quantization representation for classical many-particle
  system.
\newblock {\em J. Phys. A: Math. Gen.}, 9(9):1465--1477, 1976.

\bibitem{CarJae}
J.~C.~Jaeger H.~S.~Carslaw.
\newblock {\em Conduction of Heat in Solids}.
\newblock Oxford University Press, 1959.

\bibitem{SHeLo11}
S.~Hellander and P.~L{\"o}tstedt.
\newblock Flexible single molecule simulation of reaction-diffusion processes.
\newblock {\em J.~Comput.~Phys.}, 230(10):3948--3965, 2011.

\bibitem{ZoWo5a}
J.~S. van Zon and P.~R. ten Wolde.
\newblock Green's-function reaction dynamics: {A} particle-based approach for
  simulating biochemical networks in time and space.
\newblock {\em J.~Chem.~Phys.}, 123(23):234910, 2005.

\bibitem{ZoWo5b}
J.~S. van Zon and P.~R. ten Wolde.
\newblock Simulating biochemical networks at the particle level and in time and
  space: {G}reen's-function reaction dynamics.
\newblock {\em Phys.~Rev.~Lett.}, 94(12):128103, 2005.

\bibitem{uRDME}
S.~Engblom, L.~Ferm, A.~Hellander, and P.~L{\"o}tstedt.
\newblock Simulation of stochastic reaction-diffusion processes on unstructured
  meshes.
\newblock {\em J. Sci. Comput.}, 31(3):1774--1797, 2009.

\bibitem{MoWe65}
E.~W. Montroll and G.~H. Weiss.
\newblock Random walks on lattices. {II}.
\newblock {\em J. Math. Phys.}, 6(2):167--181, 1965.

\bibitem{Montroll68}
E.~W. Montroll.
\newblock Random walks on lattices {III}: Calculation of first-passage times
  with application to exciton trapping on photosynthetic units.
\newblock {\em J. Math. Phys.}, 10(4):753--765, 1969.

\bibitem{CollinsKimball}
F.~C. Collins and G.~E. Kimball.
\newblock Diffusion-controlled reaction rates.
\newblock {\em J. Colloid Sci.}, 4(4):425--437, 1949.

\bibitem{FBSE10}
D.~Fange, O.~G. Berg, P.~Sj{\"o}berg, and J.~Elf.
\newblock Stochastic reaction-diffusion kinetics in the microscopic limit.
\newblock {\em Proc. Natl. Acad. Sci. USA}, 107(46):19820--19825, 2010.

\bibitem{AgmonSzabo}
N.~Agmon and A.~Szabo.
\newblock Theory of reversible diffusion-influenced reactions.
\newblock {\em J. Chem. Phys.}, 92(9):5270, 1990.

\end{thebibliography}
\end{document}